\documentclass[11pt]{article}
\usepackage{amsfonts}
\usepackage{amscd}
\usepackage{amsmath}
\usepackage{epsfig}
\usepackage{amsthm}
\usepackage{amssymb}
\usepackage{amsbsy}
\usepackage{latexsym}
\usepackage{eucal}
\usepackage{mathrsfs}
\usepackage[all, knot]{xy}
\xyoption{arc} \xyoption{color}\xyoption{line}

\newtheorem{definition}{Definition}

\newtheorem{proposition}{Proposition}
\newtheorem{corollary}{Corollary}
\newtheorem{remark}{Remark}
\newtheorem{lemma}{Lemma}
\newtheorem{prop}{Proposition}
\newtheorem{lem}{Lemma}
\newtheorem{cor}{Corollary}
\newtheorem{Remark}{Remark}

\def\be{\begin{equation}}
\def\ee{\end{equation}}
\def\ben{\begin{displaymath}}
\def\een{\end{displaymath}}
\def\baa{\begin{eqnarray}}
\def\eaa{\end{eqnarray}}

\def\ba{\begin{array}}
\def\ea{\end{array}}

\newcommand{\Id}{\operatorname{Id}}
\renewcommand{\det}{\operatorname{det}}

\makeatletter
\@addtoreset{equation}{section}
\makeatother

\newcommand{\Tr}{\operatorname{Tr}}

\newtheorem{thm}{Theorem}

\def\be{\begin{equation}}
\def\ee{\end{equation}}
\def\ben{\begin{displaymath}}
\def\een{\end{displaymath}}
\def\baa{\begin{eqnarray}}
\def\eaa{\end{eqnarray}}

\def\ba{\begin{array}}
\def\ea{\end{array}}

\def\C{{\mathbb C}}
\def\Z{{\mathbb Z}}

\def\Pl{\mathbb{P}^1}
\def\2x2{{\left(\!\!\begin{array}{cc}a&b\\c&d\\\end{array}\!\!\right)}}
\def\del{\partial}
\def\f{\frac}

\newcommand{\bfm}{\mathbf{m}}
\newcommand{\tbfm}{\tilde{\mathbf{m}}}
\newcommand{\dom}{\mbox{dom}}
\newcommand{\Green}{\mathcal{G}}
\newcommand{\dis}{\displaystyle}
\newcommand{\und}{{\frac{1}{2}}}

\newcommand{\ma}{\mathrm{max}}
\newcommand{\ov}{\overline}
\newcommand{\ovz}{\overline{z}}
\newcommand{\ovw}{\overline{w}}

\newcommand{\ovl}{\overline{\lambda}}
\newcommand{\pw}{\partial_w}
\newcommand{\povw}{\partial_{\overline{w}}}

\newcommand{\povz}{\partial_{\overline{z}}}

\newcommand{\dia}{\flat}

\newcommand{\eps}{\varepsilon}

\newcommand{\spec}{\rm spec}
\newcommand{\alp}{\alpha}

\newcommand{\R}{{\mathbb{R}}}
\newcommand{\Ibb}{{\mathbb{I}}}
\newcommand{\Gbb}{{\mathbb{G}}}
\newcommand{\Nbb}{{\mathbb{N}}}
\newcommand{\Ncal}{{\mathcal{N}}}
\newcommand{\Hcal}{{\mathcal{H}}}

\newcommand{\detF}{\det_{\mathrm{Fr}}}

\begin{document}
\title{Moduli spaces of meromorphic functions and determinant of Laplacian}

\author{Luc Hillairet\footnote{{\bf E-mail: luc.hillairet@univ-orleans.fr}}, Victor Kalvin\footnote{{\bf E-mail: vkalvin@gmail.com}}, Alexey Kokotov\footnote{{\bf E-mail: alexey.kokotov@concordia.ca}} }

\maketitle

\vskip0.5cm
\begin{center}
MAPMO (UMR 7349 Universit\'e d'Orl\'eans-CNRS)
UFR Sciences, B$\hat{a}$timent de math\'ematiques
rue de Chartres,
BP 6759
45067 Orl\'eans Cedex 02
\end{center}

\vskip0.5cm
\begin{center}
Department of Mathematics and Statistics, Concordia
University, 1455 de Maisonneuve Blvd. West, Montreal, Quebec, H3G
1M8 Canada \end{center}

\vskip2cm
{\bf Abstract.}
The Hurwitz space is the moduli space of pairs $(X,f)$ where $X$ is a compact Riemann surface and $f$ is a meromorphic function on $X$. We study the Laplace operator $\Delta^{|df|^2}$ of the flat singular Riemannian manifold $(X,|df|^2)$. We define a regularized determinant for $\Delta^{|df|^2}$ and study it as a functional on the Hurwitz space. We prove that this functional is related to
a system of PDE which admits explicit integration. This leads to an explicit expression for the determinant of the Laplace operator in terms of the basic objects on the underlying Riemann surface (the prime form, theta-functions, the canonical meromorphic bidifferential) and the divisor of the meromorphic differential $df.$
The proof has several parts that can be of independent interest. As an important intermediate result we prove a decomposition formula of the type of Burghelea-Friedlander-Kappeler for the determinant of the Laplace operator on flat surfaces with conical singularities and Euclidean or conical ends. We introduce and study the  $S$-matrix, $S(\lambda)$, of a surface with conical singularities as a function of the spectral parameter $\lambda$ and relate its behavior at $\lambda=0$ with the Schiffer projective connection on the Riemann surface $X$. We also prove variational formulas for eigenvalues of the Laplace operator of a compact surface with conical singularities
when the latter move.

\vskip2cm

\section{Introduction}

\subsection{General part}

Studying the determinants of Laplacians on Riemann surfaces is motivated by needs of quantum field theory (in connection with various partition functions) and geometric analysis (in particular, in connection with Sarnak program, \cite{Sarnak}). The explicit expressions for the determinant of the Laplacian in
the metric of constant negative curvature (\cite{DH-Ph}) and in the Arakelov metric (obtained in \cite{Alvarez} in relation to so-called bosonization formulas from the string theory) for compact Riemann surfaces of genus $g>1$ are among the most spectacular results of the subject. According to Sarnak program, these determinants (which are functions on the moduli space of Riemann surfaces) can be used to study the geometry of the moduli space via  methods of Morse theory. In particular, their behavior at the boundary of moduli space is of great importance and was intensively studied (see, e. g., \cite{Wolpert}, \cite{Wentworth}).

It seems very interesting to consider the case which is in a certain
sense opposite to the case of the metric of constant curvature:
instead of distributing the curvature uniformly along the Riemann
surface $X$ one can concentrate it at a finite set  $\{P_1, \dots,
P_M\}\subset X$. This leads to a flat metric $\bfm$ on $X$ with
singularities (e. g. conical) at $P_k$. Determinants of Laplacians for
various classes of singular flat metrics were introduced and studied
at least on the formal level (via path integrals) by physicists 
(\cite{Sonoda},\cite{Zamolodchikov}, \cite{Knizhnik},
\cite{BershRadul}) and certain explicit expressions for them were
produced (see, e. g. , \cite{Sonoda}, \cite{Zamolodchikov}).

 One of the main challenges is to study such determinants from the
 spectral theory of self-adjoint operators and perturbation theory
 point of view. This was done in the mathematical literature for the
 determinants of the Laplacians of smooth metrics, in particular,
 those two mentioned above (see, e. g., Fay's book \cite{Fay} for complete compendium and
consistent exposition). The standard definition of the
determinant uses the $\zeta$-function of the corresponding Laplace operator
 \begin{equation}\label{eq1}
 \ln {\rm det}\Delta^{\bf m}=-\zeta'_{\Delta^{\bf m}}(0)\,,
 \end{equation}
 \begin{equation}\label{eq2}
 \zeta_{\Delta^{\bf m}}(s)=\sum_j \frac{1}{\lambda_j^s}\,,
 \end{equation}
where in the latter expression the sum is extended over all non-zero
eigenvalues of $\Delta^{\bfm}.$  This definition makes sense when the
metric ${\bf m}$ have conical singularities provided a regular
self-adjoint extension is considered, for instance the Friedrichs one
(see \cite{HK} for the definition of regular) . Indeed, the Laplace operator $\Delta^{\bf m}$ (with natural domain consisting of smooth functions on $X$  supported outside
 the conical points $P_k$) is not essentially self-adjoint. This fact is never mentioned by the physicists 
  and it is not clear whether this issue has been addressed. Comparing the determinants of the different self-adjoint extensions of $\Delta^{\bf m}$ leads to a nice application of Birman-Krein theory and is done in \cite{HK} (see also the references therein).
In what follows we consider the Friedrichs extension of $\Delta^{\bf m}$,  our results refer to this self-adjoint extension only.

In \cite{KK-DG} it was found an explicit expression for the
determinant of (the Friedrichs extension of) the Laplace operator
corresponding to a flat conical metric $\bfm$ with {\it  trivial
  holonomy}. Any metric of this type can be written $|\omega|^2$, where $\omega$ is a holomorphic one-form on $X$,  zeros of $\omega$  of multiplicity $\ell$ are the conical points of the metric $|\omega|^2$ with conical angle $2\pi(\ell+1)$. The moduli space of pairs $(X, \omega)$, where $X$ is a compact Riemann surface and $\omega$ is a holomorphic one-form on $X$, is stratified according to the multiplicities of $\omega$ (see \cite{KZ}).  In \cite{KK-DG} it was proved that on each stratum of the moduli space of holomorphic differentials  the ratio $$\frac{{\rm det} \Delta^{|\omega|^2}}{{\rm Area}(X){\rm det}\,\Im {\mathbb B}}\,$$ where  ${\mathbb B}$ is the matrix of $b$-periods of the Riemann surface $X$, coincides with the modulus square of a holomorphic function $\tau$ on the stratum. This holomorphic function $\tau$ (the so-called Bergman tau-function on the space of holomorphic differentials) admits explicit expression through theta-functions, prime-forms, and the divisor of the holomorphic one-form $\omega$. In the case $g=1$ the holomorphic one-forms have no zeroes, the metric $|\omega|^2$ is smooth, and the corresponding result coincides with the classical Ray-Singer formula for the determinant of the Laplacian on an elliptic curve with flat conformal metric.

In \cite{K-PAMS} it was found a comparison formula (an analog of classical Polyakov formula) relating determinants of the Laplacians in two conformally equivalent flat conical metrics. This lead to the generalization of the results of \cite{KK-DG} to the case of arbitrary flat conformal metrics with conical singularities.

Together with determinant of the Laplacians in flat conical metrics given by the modulus square of the holomorphic one form (these metrics have finite volume and the spectra of the corresponding self-adjoint Laplacians are discrete) in physical literature appear determinants of the Laplacians corresponding to flat metrics $|\omega|^2$, where $\omega$ is now a {\it meromorphic} one form on $X$. Depending on the order of the poles of $\omega$, the corresponding non compact Riemannian manifold $(X, |\omega|^2)$ of the infinite volume has cylindrical, Euclidean, or conical ends. The spectrum of the corresponding Laplace operator is continuous (with possible embedded eigenvalues, say, in case of cylindrical ends) and the Ray-Singer regularization of the determinant (\ref{eq1}, \ref{eq2}) is no longer applicable.
The way to regularize such determinants is, in principle, also
well-known (see, e. g., \cite{Muller}): considering the Laplacian
$\Delta$ as a perturbation of some properly chosen "free" operator
$\mathring\Delta$, one introduces a relative determinant ${\rm det}(\Delta, \mathring\Delta)$ in terms of the relative $\zeta$-function
\be\label{111}
\zeta(s; \Delta, \mathring\Delta)=\frac{1}{\Gamma(s)}\int_0^\infty {\rm Tr} (e^{-\Delta t}-e^{-\mathring\Delta t})t^{s-1}\,dt,
\ee
 where a suitable regularization of the integral is made (being understood in the conventional sense the integral is usually divergent for any value of $s$).

Following this approach, in \cite{HKK} we studied the regularized
determinants $$\ {\rm det}(\Delta, \mathring\Delta)=e^{-\zeta'(0;
  \Delta, \mathring\Delta)}$$ of the Laplacians on the so-called
Mandelstam diagrams - the flat surfaces with cylindrical ends (more
precisely, Riemann surfaces $X$ with the metric $|\omega|^2$, where
$\omega$ is a meromorphic one-form on $X$ having only simple poles and
such that all the periods of $\omega$ are pure imaginary and all the residues of $\omega$ at the poles are real).

In the present paper we consider determinants of the Laplacian corresponding to flat metrics with even wilder singularities: the corresponding Riemannian manifolds have Euclidean (i. e. isometric to a vicinity of the point at infinity of the Euclidean plane) and/or  conical ends (i. e. isometric to a vicinity of the point at infinity of a straight cone).
These metrics are given as the modulus square of the differential of an arbitrary meromorphic function $f$ on a compact Riemann surface $X$.
The moduli space of pairs $(X, f)$ is called the Hurwitz space ${\cal H}$. We define and study the regularized determinant of the Laplace operator corresponding to the metric $|df|^2$ as a functional on  ${\cal H}$. The main result of this work is an explicit formula for the determinant.

 It should be mentioned that  such determinants appeared for the first time in \cite{Zamolodchikov}, \cite{BershRadul} (see also \cite{Knizhnik}), although no attempt was made to define them rigorously.

\subsection{Results and organization of the paper}
Let $X$ be a Riemann surface and $f$ be a meromorphic function $f \,:\,X\rightarrow \Pl.$ The metric $|df|^2$ gives to $X$ a structure of a
(non-compact) flat Riemannian manifold with conical singularities and conical (or Euclidean) ends. The conical singularities are located at the critical points $P_1,\dots,P_M$
of $f$ (or, equivalently, at the zeros of the meromorphic one-form $df$), the ends of $X$ are located at the poles of  $df$.  The moduli space of (equivalence classes) of such pairs $(X,f)$ is known
as Hurwitz space $\Hcal$ and the critical values $z_m$, $z_m=f(P_m)$, $m=1, \dots, M$,  locally parametrize $\Hcal$.

Given such a Riemannian manifold $(X,|df|^2)$, we introduce the reference manifold $(\mathring{X},\mathring{\bfm})$ as the disjoint union of
the complete cones corresponding to the ends of $(X,|df|^2)$.  By $\Delta$  and $\mathring{\Delta}$ we denote the Laplace operators on  $(X,|df|^2)$ and $(\mathring{X}, \mathring{\bfm})$  correspondingly.

 The first part of the paper aims at defining the relative zeta-regularized determinant $\det_\zeta(\Delta, \mathring{\Delta})$ and
proving a version of the Burghelea-Friedlander-Kappeler (BFK in what follows) gluing formula (see \cite{BFK}).
This new BFK type formula is a generalization of the Hassell-Zelditch  formula for the determinant of the Laplacian in exterior domains~\cite{Hassell-Zelditch}; here we
rely  on ideas from~\cite{Carron1,Carron2,Hassell-Zelditch}.

In order to obtain the gluing formula, we   cut $X$ along some hypersurface $\Sigma.$ This decomposes $X$ into a
compact part $X_-$ and the union of conical/Euclidean ends $X_+$. The latter is isometric to the reference surface $\mathring{X}$ with a compact part
$\mathring{X}_-$ removed. There is some latitude in choosing the initial $\Sigma$. In order to choose $\Sigma$ we first specify some large $R$ and then in each end of $(X,|df|^2)$ we take a circle whose radius depends on R and on the cone angle of the end, see Definition~\ref{def:Sigma}. As expected the gluing formula then involves the Neumann jump operator $\Ncal$ on $\Sigma$ and the Dirichlet Laplacian $\Delta^D_-$ on $X_-$, see Theorem~\ref{Mintro} below.
 \begin{thm} \label{Mintro}
For $R$ large enough we have the BFK gluing formula
$$
\det_\zeta (\Delta,\mathring\Delta)= C \det^*_\zeta \mathcal N\cdot\det_\zeta\Delta^D_-,
$$
where $\mathcal N$, $\Delta_-^D$ depend on $R$. The constant $C$ depends on $R$ but not on the moduli parameters
$z_1,\dots z_M$ as long as the corresponding critical points $P_m$ do not approach $\Sigma.$
\end{thm}

 Note that the proof of the gluing formula also holds for a more general class of metric (see Remark \ref{rk:mgm}).

Let us now sketch some steps leading to this Theorem.
First we start from the BFK gluing formula for $\det_\zeta(\Delta-\lambda,\mathring\Delta-\lambda)$ obtained in~\cite{Carron1} for negative (regular) values of the spectral parameter $\lambda$.
In order to obtain a gluing formula for $\det_\zeta(\Delta,\mathring\Delta)$,  we study the behaviour of all ingredients in the gluing formula for  $\det_\zeta(\Delta-\lambda,\mathring\Delta-\lambda)$ as $\lambda\to0-$
(i.e. at the bottom of the continuous spectrum of $\Delta$ and $\mathring\Delta$)
and then pass to the limit. As usual, this essentially reduces to derivation of asymptotics as $\lambda\to0-$ for the zeta regularized determinant of the  Neumann jump operator and for the spectral shift function of the pair $(\Delta,\mathring\Delta)$.  In principle both asymptotics were obtained in~\cite{Carron2} for  Schr\"odinger type operators on manifolds with conical ends. Unfortunately those asymptotics cannot be used for our purposes because  the asymptotic for the Neumann jump operator contains an unspecified constant and the asymptotic for the spectral shift function is not sufficiently sharp.   We demonstrate that at least in our setting (no potential) the methods of~\cite{Carron2} can be improved to specify the constant and to obtain a sufficiently sharp asymptotic of the spectral shift function as needed for the proof of our BFK formula. Once
these asymptotics are obtained, we follow the lines of~\cite{Hassell-Zelditch} in our study of the behaviour of  $\det_\zeta(\Delta-\lambda,\mathring\Delta-\lambda)$ as $\lambda\to0-$
and also in definition of $\det_\zeta(\Delta,\mathring\Delta)$.

Using this BFK formula, we prove (as it was done in similar situations in \cite{K-PAMS}, \cite{HKK}) that the variations of the determinant of the Laplacian with respect to the moduli parameters $z_k$  remain the same if we replace the metric $\bfm=|df|^2$ of infinite volume by a metric $\tbfm$ of finite volume, where  $\tbfm$ coincides with $\bfm$ outside vicinities of the poles of $f$ and with some standard nonsingular metric of finite volume inside these vicinities. The aim of the second part of the paper is thus to study the zeta-regularized determinant of this new metric $\tbfm$
and its variation with respect to moduli parameters.

We show that these variations can be conveniently expressed using the so-called $S$-matrix, so we start the second part of the paper by introducing this object and deriving
some of its properties. We think that the $S$-matrix is  an important characteristic of a compact Riemann surface $X$ equipped with a conformal metric $\tbfm$ with conical
singularities. It is introduced in analogy with scattering problems and the general theory of boundary triples (see \cite{Grubb}).
Basically, the $S$-matrix is a meromorphic  matrix function of the spectral parameter $\lambda$. It serves as a formal analog of the scattering matrix on complete non-compact manifolds (say, on  manifolds with cylindrical ends) in the sense that the elements of $S$ are coefficients in asymptotics of certain eigenfunctions   (all growing terms in the asymptotics near conical singularities of   $\tbfm$ are   interpreted as incoming waves and all decaying terms --- as outgoing). Let us also note that there is no true scattering on the (incomplete) manifold $(X, \tbfm)$  (corresponding operators have no continuous spectrum) and that $S(\lambda)$ is a non-unitary matrix with poles at the eigenvalues $\lambda$ of the Friedrichs self-adjoint Laplacian on  $(X, \tbfm)$~\cite{HK}.

It turns out that the value, $S(0)$, of the $S$-matrix can be found
explicitly: it depends only on the conical angle at the conical point,
the conformal class of the surface $X$, and  the choice of the
holomorphic local parameter near the conical point (the so-called
distinguished local parameter for the metric $\tbfm$). For instance,
when the conical angle is $4\pi,$ we express the matrix elements of
$S(0)$ through the Bergman reproducing kernel for holomorphic
differentials and a certain special projective connection on $X$ (the
so-called Schiffer projective connection).

In the general case we prove that a certain  linear combination of the
matrix elements of $S(0)$ (actually, the one that appears later in the
variational formulas for the determinant) can be expressed as the
derivative (of the order depending on the conical angle) of the
Schiffer projective connection.


We continue by studying the moduli variations of the zeta-regularized determinant of $\Delta^{\tbfm}.$ We use the Kato-Rellich perturbation
theory to compute the variation of individual eigenvalue branches and then a contour argument similar to the one from \cite{HK} to get the variational formula
for the determinant. This formula involves a combination of the matrix elements of $S(0)$ and hence the Schiffer projective connection.
Writing it into an invariant form, we obtain  the following theorem.

\begin{thm}\label{thm:introvardet}
Let $P_m$ be a zero of the meromorphic differential $df$ of
multiplicity $\ell_m$ and let $z_m=f(P_m)$ be the corresponding
critical value of $f$. Let the metric $\tbfm$ be obtained from the
metric $|df|^2$  via smoothing of the conical ends.
Then
\begin{equation}
\partial_{z_m}\ln \frac{{\rm det}_\zeta^* (\Delta^{\tbfm})}{{\rm det}\Im {\mathbb B}}=-\frac{1}{12\pi i}\oint_{P_m}\frac{S_B-S_f}
{df}\,
\end{equation}
where $S_B$ is the Bergman projective connection, $S_f=\frac{f'''f'-\frac{3}{2}(f'')^2}{(f')^2}$ is the Schwarzian derivative, and $\mathbb B$ is the matrix of $b$-periods.
\end{thm}

In this theorem we can replace $\det_\zeta^*(\Delta^{\tbfm})$ by $\det
(\Delta,\mathring{\Delta})$ since we  proved before that the moduli
variations of both functions coincide.

The system of PDE for  $\det\Delta^{\tbfm}$ that appears in Theorem \ref{thm:introvardet}  is the governing system for the Bergman tau-function on the Hurwitz space (introduced and studied in \cite{KK-MPAG}, \cite{KK-DG}, \cite{KKZ}, \cite{KS}).
The latter system was explicitly integrated in \cite{KK-IMRN} and in \S5 we remind this result (unfortunately, technically involved). This leads to the following
explicit formula for $\det_\zeta (\Delta, \mathring\Delta )$.

\begin{thm}\label{thm:introfinal}
Let $(X, f)$ be an element of the Hurwitz space $\Hcal(M, N)$ and let $\tau(X, f)$ be given by expressions (\ref{E0}, \ref{E1}, \ref{Eg}).
There is the following explicit expression for the regularized relative determinant of the Laplacian $\Delta$ on the Riemann surface $X$:
\begin{equation}
{\rm det}_\zeta (\Delta,\mathring\Delta)=C\,{\rm det}\Im {\mathbb B}\,|\tau|^2\,,\end{equation}
 where $C$ is a constant that depends only on the connected component of the space $\Hcal(M, N)$ containing the element $(X, f)$.
\end{thm}

We finish the paper with two illustrating examples in genus $0$, deriving the formulas for the determinant of the Laplacian on the space of polynomials of degree $N$ and on the space of rational functions with three simple poles.

{\bf Acknowledgements.} This work was finished during the stay of the third author (A. K.) at the Max Planck Institute of Mathematics (Bonn). A. K. thanks the Institute for hospitality and excellent working conditions. The work of L.H  is partly supported by the ANR program Gerasic-ANR-13-BS01-0007-0. The work of A. K. is supported by NSERC and ANR program Gerasic-ANR-13-BS01-0007-0.

\section{The regularized determinant as a functional on Hurwitz space and a BFK gluing formula}
To a pair $(X,f)$, where $X$ is a compact Riemann surface and $f$ is a meromorphic
function on $X$ (i. e. to an element of the Hurwitz space), there corresponds a Riemannian manifold $(X, |df|^2)$.
Our aim is to define a regularized determinant of the corrresponding Laplacian and to prove a BFK-type gluing formula.
Since the metric $\bfm=|df|^2$ has conical singularities and non-compact conical ends, this is not that straightforward and requires several steps.
First, we consider regular values of the spectral parameter $\lambda^2$, i.e.  $\lambda^2\in \C\setminus [0,\infty)$.
In that case, the definition of the relative determinant and the BFK-gluing formula are the same as in \cite{Carron1}.
Then we derive estimates for the determinant of the Dirichlet-to-Neumann operator as $\lambda$ approaches $0$; our
methods here are closely related to those of~\cite{Carron2}.  These estimates allow us to
define a zeta-regularized determinant for $\lambda=0$ similarly to~\cite{Hassell-Zelditch}. Then we prove the gluing formula for thus defined determinant.
At the end of this section we use the gluing formula to compactify $(X,|df|^2)$ in such a way, that
locally, the moduli variations remain the same.

\subsection{The flat Laplacian of an element in Hurwitz space}
We will be dealing with conical singularities and conical ends. These are defined in the following way.

\begin{definition}\hfill \\
\vspace{-11pt}
\begin{itemize}
\item For any $\ell  \in \Nbb$ the Euclidean cone of total angle $2\ell\pi$ is the Riemannian manifold $(\C, |\ell y^{\ell-1}dy|^2).$
\item A point $P$ in a Riemannian manifold will be a conical singularity of angle $2\ell \pi$ if there is a neighbourhood of $P$ that is isometric to
the set $\left(\{|y|<\eps\},\,|\ell y^{\ell-1}dy|^2\right)$ for some positive $\eps.$
\item A open set $\Omega \subset X$ of a Riemannian manifold $(X,\bfm)$ such that $(\Omega,\bfm)$ is isometric to $\left( \{|y|>R\},\,|\ell y^{\ell-1}dy|^2\right)$
for some positive $R$ will be called a conical end of angle $2\ell \pi$ (Euclidean end if $\ell=1$).
\end{itemize}
\end{definition}

Let $(X,\bfm)$ be a Riemannian manifold such that the metric is flat with a finite number of conical singularities and conical ends. Denote by
$\Delta$ the self-adjoint  Friedrichs extension of the (non-negative) symmetric Laplace operator defined on  smooth compactly supported functions that vanish near the conical singularities.

 Let $f$ be a meromorphic function on a compact Riemann surface $X$ of genus $g\geq 0$ or, what is the same, a ramified covering of the Riemann sphere
 \begin{equation}\label{covering}
  f: X\to \Pl.
 \end{equation}

Two coverings $f_1:X_1\to P^1$ and $f_2:X_2\to P^1$ are called equivalent if there exists a biholomorhic map $g:X_1\to X_2$ such that $f_1=f_2\circ g$.

The following constructions are standard, we recall them for convenience of the reader.

The critical points, $P_m$, $m=1, \dots, M,$ of the function $f$  (i. e. those points for which $df(P_m)=0$) are the ramification points of the covering, the points $z_m=f(P_m)$ are called the critical values. The ramification index of the covering at the point $P_m$ equals to $\ell_m+1$, where $\ell_m$ is the order of the zero of the one-form $df$ at $P_m$. Denote by $\infty_1, \dots, \infty_K$ the poles of $f$, and let $k_1,\dots, k_K$ be their multiplicities.

Then the covering (\ref{covering}) has degree $N=k_1+\dots\,+ k_K$ and the following Riemann-Hurwitz formula holds:

$$\sum_{m=1}^M \ell_m-\sum_{j=1}^{K}(k_j+1)=2g-2\,,$$
where $g$ is the genus of $X.$

Pick some regular value $z_0\in \Pl$ and draw on $\Pl$ the segments $I_0:= [z_0,\infty], ~I_m=[z_0,z_m],~m=1,\dots,M$.
It may happen that some segment is repeated several times if different critical points take the same critical value.
We may also choose $z_0$ such that all these segments have pairwise disjoint interiors.
Denote by $L:= \bigcup_{m=0}^M I_m$ the union of these segments and
observe that $\Pl\setminus L$ contains only regular values of $f$. It follows that $X\setminus f^{-1}(L)$, the complement of the preimage of $L$ by $f$ in $X$ has $N$
connected components. By construction $f$ is a biholomorphic map from each of these connected components onto $\Pl \setminus L.$ We denote these
connected components by $\C_n,\,n=1\dots N$ and call them the sheets
of the covering. Each $\C_n$ can be seen as a copy of the complex plane equipped with the cuts provided by $L.$

%
%

On each sheet, the metric $|dz|^2$ lifted from the base $\Pl$ to the covering space coincides with the metric $|df|^2$. The Riemannian manifold $(X, |df|^2)$ is thus obtained by gluing $N$ copies of a Euclidean plane (${\mathbb C}, |dz|^2$) with a system of
non-intersecting cuts, one of which extends to infinity.

For each critical point $P_m$ of ramification index $\ell_m+1$, we obtain a $(\ell_m+1)$-cycle $\gamma_j$ obtained by looking in
which order the sheets are following one another when making a small loop around $P_m.$ It follows that $P_m$
is a conical singularity of angle $2\pi(\ell_m+1).$

For each $z_m,~ m\neq 0$, we obtain a permutation in  $S_N$ by composing the cycles for each critical point
in $f^{-1}(z_m).$ We thus obtain $M'$ permutations $\sigma_{m'},~m'=1\dots M'$, where $M'$ is the number of different critical values.
The critical value $z_m$ is thus associated with one cycle in one of the permutations $\sigma_{m'},\, m'=1\dots M'.$

In the same manner we obtain a permutation $\sigma_0$ by looking at
the preimage of a large loop that surrounds $z_0$ (or equivalently, a
small loop around $\infty$ in the base $\Pl$).
This permutation describes the structure at infinity of the Riemannian manifold $(X,|df|^2)$ :
each fixed point of $\sigma_0$ corresponds to a flat Euclidean end and a cycle of length $k$ to a conical end of angle $2k\pi.$
A pole in $f$ of order $k$ corresponds to a conical end of angle $2k\pi$ (and therefore a Euclidean end for a simple pole).

The flat structure on $(X,|df|^2)$ is completely characterized by the positions of the critical values $z_m, \, m=1\dots M$, and by the
permutations $\sigma_{m'}, \,m'=0\dots M'$. Conversely, starting from $M'+1$ permutations of $S_N,$ and $M'$ distinct points $w_1,\dots w_{m'}$
in $\C,$ we construct the sequence $z_0,\dots z_M$ by choosing a distinct point $z_0$ and, for $m>0$, by repeating $w_{m'}$ as many times as there are disjoint
cycles in $\sigma_{m'}.$ We then glue $N$ sheets according to the scheme prescribed by the  permutations and
 obtain a (not necessarily connected) flat surface $(X,\bfm)$ with conical singularities and conical ends.

It turns out that it is always possible to find a meromorphic function $f$ from $X$ to $\Pl$ such that $(X,\bfm)$ is isometric to $(X,|df|^2)$.

Introduce the Hurwitz space $\Hcal(N, M)$ of equivalence classes of
coverings $f: X\to P^1$ of  degree $N$ with  $M$ ramification points
of (fixed) indices $\ell_1+1$, $\dots$, $\ell_M+1$ and $K$ poles of
(fixed) multiplicities $k_1, \dots, k_\ell$; $k_1+\dots, k_K=N$. The
space $\Hcal(N, M)$ is a complex manifold of dimension $M$ (see
\cite{Fulton}, we notice here that it may have more than one connected
components) and the critical values $z_1, \dots, z_M$ can be taken as
local coordinates on $\Hcal(N, M)$.

If all the critical points of the maps $f$ are simple, then the corresponding Hurwitz space is usually denoted by
$\Hcal_{g, N}(k_1, \dots, k_K)$ and is known to be connected (see \cite{Natanzon}).

\begin{definition}\label{def:moduli}
We will refer to the coordinates $z_1, \dots, z_M$ as {\it moduli}.
\end{definition}

From the flat metric point of view, moving $z_m$ can be  easily realized by cutting a small ball around $P_m$, then move $P_m$ inside this
ball. Since the boundary of the ball does not change we can  glue the new ball back into the surface.


For such a Riemannian manifold $(X,|df|^2)$ we define a reference manifold $(\mathring{X},\mathring{\bfm})$ which is obtained in the following way.
Take those $N$ sheets with cuts that correspond to $X,$ and in the gluing scheme of $X$, keep $\sigma_0$ and replace all the permutations $\sigma_{m'}$, $m'>0$ by
the identity. It can be easily seen that $(\mathring{X}, \mathring{\bfm})$ is the disjoint union of cones, those cones  correspond to the conical ends
of $X$ and the tip of each cone is now located above $z_0.$

The Laplacian $\Delta$ can be considered as a perturbation of the free Laplacian $\mathring\Delta:=\Delta^{\mathring{\bfm}}$ acting in $L^2(\mathring{X})$. The perturbation is basically reduced to the change of the domain of the unbounded operator:
when we make slits on $\mathring{X}$ and glue them according to a certain gluing scheme, it induces boundary conditions on the sides of the cuts.
The determinant of $\Delta$ will then be defined in terms of the relative zeta function~\eqref{111} as a regularized relative determinant $\det_\zeta (\Delta, \mathring\Delta)$.

The main goal of this work is to study the relative determinant $\det_\zeta (\Delta, \mathring\Delta)$
as a functional on the space $\Hcal(N, M)$.


\subsection{Relative Determinant and BFK gluing formula for negative energies}

Let $X$ be a compact Riemann surface and let $f$ be a nonconstant meromorphic function on $X$. Introduce the flat singular metric $\bfm=|df|^2$ on $X$.
As it is explained in the previous section, the flat singular Riemannian manifold $(X, \bfm)$ has conical points (at the zeros, $P_1, \dots, P_M$, of the differential $df$) and
conical ends of angle $2\pi k_j$ at the poles, $\infty_1, \dots, \infty_K$, of $f$, where $k_j$ is the order of the corresponding pole. Let $\Delta$ be the
(Friedrichs) Laplacian on $(X,\bfm)$.

Let $(\mathring{X},\mathring{\bfm})$ be the reference unperturbed manifold and  let $\mathring{\Delta}$ be the associated Laplace operator.
We recall that $(\mathring{X},\mathring{\bfm})= \bigcup_{j=1}^K (\C, |k_jy_j^{k_j-1} dy_j|^2).$

Since $(X,\bfm)$ and $(\mathring{X},\mathring{\bfm})$ are isometric outside a compact region the methods and results of \cite{Carron2} apply.

For $R>0$ large enough, there is a  subset $X_+(R)\subset X$ that is isometric to
\be\label{X+}
\cup_{j=1}^K\{y_j\in\Bbb C : |y_j|\geq R ^{1/k_j}\}\subset \mathring{X}.
\ee

\begin{definition}\label{def:Sigma}
We denote by $\Sigma_R$ the boundary of the region $X_+(R).$
It is the union of $K$ circles $\{y\in\Bbb C\,:|y|=  R^{1/k_j}\}$ on $X.$
\end{definition}

Note that $R$ will be chosen at the very beginning of construction and will then be fixed. In what follows we omit the reference to $R$ and simply write $\Sigma,~ X_+$.

We represent $X$ in the form
$$
X=X_-\cup_\Sigma X_+,
$$
where  $X_-= X\setminus (X_+\cup\Sigma)$.

Following \cite{Carron1} we first define the external Dirichlet-to-Neumann operator.
We consider each conical end  $\{y_j\in\Bbb C_j : |y_j|\geq R ^{1/k_j}\}$ of $X_+$ separately and omit the subscript $j$ for brevity of notation.
 We introduce the  coordinates $(r, \varphi)$ where $r=|y|^k\in [R,\infty)$  and  $\varphi=\arg y\in (-\pi,\pi]$. We have
$$
g= dr^2+k^2 r^2\,d\varphi^2,\quad \Delta=
r^{-2}\bigl((r\partial_r)^2+k^{-2}\partial_\varphi^2\bigr).
$$
Separation of variables  shows that for $\lambda\in\Bbb C\setminus\{0\}$ with $\Im\lambda\geq 0$ the exterior Dirichlet problem
\be\label{DPext}
(\Delta-\lambda^2)u(\lambda)=0\text { on } X_+,\quad u(\lambda)=f\text{ on } \Sigma,
\ee
has a unique solution of the form
$$
u(r,\varphi;\lambda)=\sum_{n=-\infty}^\infty C_n \frac {H^{(1)}_{\nu_n}(\lambda r)}{ H^{(1)}_{\nu_n}(\lambda R) }e^{in\varphi} ,\quad \nu_n=\frac{|n|}{kR},$$
where $f\in C^\infty(\Sigma)$, $C_n=(2\pi)^{-1}\int_{-\pi}^{\pi} f(\varphi) e^{-i n \varphi } \,d\varphi$, and $H^{(1)}_n$ is the Hankel function.
This solution is in $L^2(X_+)$ if $\Im\lambda>0$. If $\Im\lambda=0,$ it is the unique outgoing solution that satisfies the Sommerfeld radiation condition
$$\sqrt {r}\bigl( \partial_r u(\lambda)-i\lambda r u(\lambda)\bigr)\to 0\text { as } r\to\infty.$$

The external Dirichlet-to-Neumann operator on $\Sigma$ acts by the formula
\be\label{Next}
\mathcal N_+(\lambda) f =-\partial_r u(\lambda)\upharpoonright_{r=R}.
\ee
Thus $\psi_n(\varphi)=(\operatorname{Vol} \Sigma )^{-1/2}e^{in\varphi}$ are eigenfunctions of $\mathcal N_+(\lambda)$ with $\|\psi_n\|_{L^2(\Sigma)}=1$, and
\be\label{eig}
\mu_n(\lambda)=\mu_{-n}(\lambda)=-\frac {\partial_r H^{(1)}_{\nu_n}(\lambda r)\upharpoonright_{r=R}}{ H^{(1)}_{\nu_n}(\lambda R) }
\ee
are the corresponding eigenvalues (if  $\Im\lambda\geq 0$, $\lambda\neq 0$).

We can also consider $\Sigma$ and $X^+$ as subsets of $\mathring{X}$. Then  in the same manner we have $\mathring{X} = \mathring{X}_- \cup_\Sigma \mathring{X}_+.$

Let $\Delta^D_{\pm}$ be the Friedrichs extensions of the Dirichlet Laplacians in $L^2(X_{\pm})$.
We denote by $\Delta^D:= \Delta^D_- \oplus \Delta^D_+ $ the Friedrichs Laplace operator on $L^2(X)$ with
Dirichlet boundary condition on $\Sigma$. Similarly, we define $\mathring{\Delta}^D_{\pm}$ and $\mathring{\Delta}^D.$

The spectrum $\spec(\Delta_-^D)$ of the  positive self-adjoint operator $\Delta_-^D$ is discrete. For any $\lambda^2\in\Bbb C\setminus\sigma(\Delta_-^D)$ and $f\in H^1(\Sigma)$ there exists a unique solution $u(\lambda)\in H^{3/2}(X_-)$ to the Dirichlet problem
\begin{equation}\label{Din}
(\Delta-\lambda^2)u(\lambda)=0\text { on } X_-\setminus\Sigma, \quad u(\lambda)=f\text{ on } \Sigma,
\end{equation}
such that
\be
\label{u}
u(\lambda)=\tilde f- (\Delta_-^D-\lambda^2)^{-1}(\Delta-\lambda^2)\tilde f,
\ee
where $\tilde f\in H^{3/2}(X_-)$ is a continuation of $f$ and
$$(\Delta_-^D-\lambda^2)^{-1}: H^{-1/2}(X_-)\to H^{3/2}(X_-)$$
is a holomorphic function of $\lambda^2\in\Bbb C\setminus\sigma(\Delta_-^D)$; here $\|u; H^s(X_-)\|=\|(\Delta_-^D)^{s/2}u; L^2(X_-)\|$. The Dirichlet-to-Neumann operator  $\mathcal N_-(\lambda)$  on $\Sigma$ acts by the formula
$$\mathcal N_-(\lambda) f={\partial_r u(\lambda)}\upharpoonright_{r=R},$$
 where $f$ is the right hand side in~\eqref{Din} and $u(\lambda)$ is defined by~\eqref{u}. The function  $\lambda^2\mapsto \mathcal N_-(\lambda)\in \mathcal B(H^1(\Sigma), L^2(\Sigma)) $  is  holomorphic  in   $\Bbb C\setminus\sigma(\Delta_-^D)$; here and elsewhere $\mathcal B(\cal X,\cal Y) $ stands for the space  of bounded operators  from $\cal X$ to $\cal Y$.

Finally, we introduce the Neumann jump operator
  $$
  \mathcal N(\lambda)=\mathcal N_+(\lambda)+\mathcal N_-(\lambda),
  $$
 which is a first order elliptic classical pseudodifferential operator on $\Sigma$ with the principal symbol $2|\xi|$.

For $\lambda^2\leq 0$ the operator $\mathcal N(\lambda)$ is formally self-adjoint and nonnegative, it is positive if $\lambda^2<0$, and $\ker\mathcal N(0)=\{c\in\Bbb C\}$  (see e.g.~\cite[Sec. 3.3]{Carron2}  for details). (Note that in Theorem~\ref{M} the operator $\mathcal N(0)$ is denoted by  $\mathcal N$.)  Let $\lambda^2<0$. The function $\zeta(s)=\Tr\mathcal N(\lambda)^{-s}$ is holomorphic in $\{s\in\Bbb C: \Re s >1\}$ and admits a meromorphic continuation to $\Bbb C$ with no pole at zero; we set $\det_\zeta \mathcal N(\lambda)=e^{-\zeta'(0)}$.

It is known~(see \cite[Theorem 2.2]{Carron1}) that the difference
$$
(\Delta+1)^{-1}-(\Delta^D+1)^{-1}
$$ is in the trace class. By the Krein theorem, see e.g.~\cite[Chapter 8.9]{yafaev} or~\cite[Theorem 3.3]{Carron1}, there exists a spectral
shift function $\xi(\cdot\,;\,\Delta,\,\Delta^D)\in L^1(\Bbb R_+, (1+\lambda^2)^{-2}\lambda\,d\lambda)$
such that
\be\label{aux}
\Tr\bigl( (\Delta+1)^{-1}-(\Delta^D+1)^{-1}
\bigr)=-\int_0^\infty \xi(\lambda\,;\,\Delta,\,\Delta^D)(1+\lambda^2)^{-2}2\lambda\,d\lambda.
\ee
Moreover, the following representation is valid
\begin{equation}\label{Krein}
\Tr\bigl( e^{-t\Delta}-e^{-t\Delta^D}
\bigr)=-t\int_0^\infty e^{-t\lambda^2} \xi(\lambda\,;\,\Delta,\,\Delta^D)2\lambda\,d\lambda,
\end{equation}
which implies that  the left hand side in~\eqref{Krein} is absolutely bounded uniformly in $t>\epsilon>0$.
 The heat trace asymptotic
 \begin{equation}\label{t0}
 \Tr \bigl(e^{-t\Delta }- e^{-t\Delta^D}\bigr)\sim \sum_{j\geq- 2} a_j t^{j/2},\quad t\to 0+,
\end{equation}
can be obtained in a usual way, see e.g.~\cite[Lemma 4]{HKK}.
Thus for $\lambda^2<0$ the relative zeta function given by
$$
\zeta(s; \Delta-\lambda^2,\Delta^D-\lambda^2)=\frac 1 {\Gamma(s)}\int_0^\infty t^{s-1} e^{\lambda^2 t } \Tr (e^{-t\Delta }- e^{-t\Delta^D})\,dt
$$
is defined for $\Re s>1$ and continues meromorphically to the complex plane with no pole at $s=0$ by the usual argument. The relative determinant is defined to be
$$
\det_\zeta(\Delta-\lambda^2,\Delta^D-\lambda^2)=e^{-\zeta'(0; \Delta-\lambda^2,\Delta^D-\lambda^2)}.
$$
By~\cite[Theorem 4.2]{Carron1} we have the gluing formula
\be\label{sur1}
\det_\zeta(\Delta-\lambda^2,\Delta^D-\lambda^2)=\det_\zeta \mathcal N(\lambda),\quad \lambda^2<0.
\ee
(Although only smooth manifolds are considered in~\cite{Carron1}, it is fairly straightforward to see that the argument in~\cite{Carron1} remains valid for~\eqref{sur1} as far as we consider only Friederichs extensions and there are no conical points on $\Sigma$.)

All the constructions above can also be done for $(\mathring{X},\mathring{\bfm})$. Thus  similarly to~\eqref{sur1} we have
\be\label{sur2}
\det_\zeta(\mathring\Delta-\lambda^2,\mathring\Delta^D-\lambda^2)=\det \mathring{\mathcal N}_\zeta(\lambda),\quad \lambda^2<0.
\ee

Observe that since all operators can be seen as acting on $L^2(\mathring{X})$ we have
\begin{equation*}
\begin{split}
e^{-t\Delta}-e^{-t\mathring{\Delta}}& = \left( e^{-t\Delta}-e^{-t\Delta^D}\right) - \left(e^{-t\mathring{\Delta}}-e^{-t\mathring{\Delta}^D}\right ) +\left( e^{-t\Delta^D}-e^{-t\mathring{\Delta}^D}\right) \\
&= \left( e^{-t\Delta}-e^{-t\Delta^D}\right) - \left(e^{-t\mathring{\Delta}}-e^{-t\mathring{\Delta}^D}\right ) +\left( e^{-t\Delta^D_-}-e^{-t\mathring{\Delta}^D_-}\right), \\
\end{split}
\end{equation*}
where for the last line we have used that $\Delta^D=\Delta^D_-\oplus \Delta_+^D,~\mathring{\Delta}^D=\mathring{\Delta}^D_-\oplus \Delta_+^D$ since
$X_+$ and $\mathring{X}_+$ are isometric.

It follows that we can take the trace of both sides and thus define the following relative zeta function for $\Re s>1$
$$
\zeta(s;\Delta-\lambda^2,\mathring\Delta -\lambda^2)=\frac 1 {\Gamma(s)}\int_0^\infty t^{s-1} e^{\lambda^2 t }  \Tr (e^{-t\Delta }- e^{-t\mathring\Delta})\,dt,\quad \lambda^2<0.
$$
Moreover,  we obtain the relation
\begin{equation*}
\begin{split}
\zeta(s;\Delta-\lambda^2,\mathring\Delta -\lambda^2)\,=&\,\zeta(s;\Delta-\lambda^2,\mathring\Delta^D -\lambda^2)\,
-\,\zeta(s;\mathring{\Delta}-\lambda^2,\mathring{\Delta}^D -\lambda^2)\\
& ~~ \,+\,\zeta(s,\Delta^D_--\lambda^2)\,-\,\zeta(s,\mathring{\Delta}^D_--\lambda^2).
\end{split}
\end{equation*}
All the functions continue meromorphically to the complex plane with no pole at $0.$
Passing to the determinant, we obtain
$$
\frac {\det_\zeta(\Delta-\lambda^2,\Delta^D-\lambda^2)}{\det_\zeta(\mathring\Delta-\lambda^2,\mathring\Delta^D-\lambda^2)}=\frac{\det_\zeta(\Delta-\lambda^2,\mathring\Delta -\lambda^2)\det_\zeta(\mathring\Delta^D_--\lambda^2)}{\det_\zeta(\Delta_-^D-\lambda^2)}.
$$

Thus dividing~\eqref{sur1} by~\eqref{sur2} we obtain
\begin{equation}\label{sg1}
\frac{\det_\zeta(\Delta-\lambda^2,\mathring\Delta -\lambda^2)\det_\zeta(\mathring\Delta^D_--\lambda^2)}{\det_\zeta(\Delta_-^D-\lambda^2)}=\frac{\det_\zeta\mathcal N(\lambda)}{\det_\zeta\mathring{\mathcal N}(\lambda)},\quad\lambda^2<0,
\end{equation}
 where $\mathring {\mathcal N}(\lambda)$ and $\mathring\Delta^D_-$ are moduli independent.

In order to take the limit $\lambda^2\to 0 -$ in~\eqref{sg1}, we will need the asymptotic behavior of all the ingredients
in the latter equation. We start with $\det_\zeta \mathcal N(\lambda).$

 \subsection{Asymptotic of $\det_\zeta \mathcal N(\lambda)$ as $|\lambda|\to 0+$, $\Im\lambda\geq 0$}
In this section we follow~\cite{Carron2}, where a similar problem is studied. Our purpose here is to refine the asymptotic of  $\det_\zeta \mathcal N(\lambda)$ from~\cite{Carron2} as needed for the proof of our gluing formula.

First we need to understand the behavior of the internal and external Dirichlet-to-Neumann operators.
Since the internal Dirichlet Laplacian $\Delta^D_-$ is positive, there is no problem in letting $\lambda$
go to $0$ in the definition of $\Ncal_-(\lambda).$

Concerning the external Dirichlet-to-Neumann operator, by separation of variables in each conical end,
we see that in the case $\lambda=0$ the exterior Dirichlet problem~\eqref{DPext} has a unique solution of the form
\be\label{oo}
u(r, \varphi;0)=\sum_{n=-\infty}^\infty C_n\Bigl(\frac{R}{r}\Bigr)^{\nu_n}e^{in\varphi},
\ee
where $f= \sum C_n e^{in\varphi}$. We recall that $\nu_n= \frac{|n|}{kR}$ and $\psi_n(\phi)= (2\pi)^{-\frac{1}{2}}e^{in\phi}$.
The external Dirichlet-to-Neumann $\Ncal_+(0)$ is obtained by applying $-\partial_r$ to this solution.Clearly, $\{|n|/(kR^2),\psi_n\}_{n=-\infty}^\infty $ is a complete set of the eigenvalues and orthonormal eigenfunctions of the operator $\mathcal N_+(0)$.

\begin{Remark} Note that thanks to the special choice of the lower bound on $y$ in~\eqref{X+} the eigenvalue $\mu_0(\lambda)$ of $\mathcal N_+(\lambda)$ corresponding to the constant eigenfunction  $\psi_0$ does not depend on $k$. This will be important  in our proof of the BFK gluing formula in the case $K\geq 1$ if $k_j\neq k_i$ for some $i,j=1,\dots,K$.
 \end{Remark}

It is convenient to present the argument in the case where $K=1$ so that $X$ has only one conical end.
We will explain afterwards how the proof is modified for $K>1.$

\subsubsection{The case $K=1$}

 In the series~\eqref{oo} only the terms with $\nu_{n}>1$ are in $L^{2}(X_+)$. As a result,  in a neighborhood of zero,  properties of $\lambda\mapsto\mathcal N_+(\lambda)$ on the eigenspaces of $\mathcal N_+(0)$ corresponding to the eigenvalues $|n|/(kR^2)>1/R$ and $|n|/(kR^2)\leq 1/R$ are essentially different.
 Consider the spectral projector ${\bf P}=\sum_{0\leq n\leq k R} P_n$ of $\mathcal N_+(0)$ on the interval $[0,1/R]$; here
$$
P_0=(\cdot,\psi_0)_{L^2(\Sigma)};\quad P_n= (\cdot, \psi_n)_{L^2(\Sigma)}\psi_n+(\cdot, \psi_{-n})_{L^2(\Sigma)}\psi_{-n}.
$$
\begin{lem}[see {\cite[Prop 4.5]{Carron2}}]   \label{L1}
 We have
 $$
\mathcal N_+(\lambda)\bigl(\Id-{\bf P}\bigr)=\Psi(\lambda^2)+\mathcal L(\lambda^2),
$$
where $\Psi(z)$ is an elliptic pseudo-differential operator of order $1$ which is a holomorphic function of $z$ in a neighbourhood of zero   and $\mathcal L(z)$ is an operator with smooth integral kernel which is a $C^1$
function of $z$ in a neighbourhood of zero with $\Im z\geq 0$.
\end{lem}

 Recall that the eigenfunctions $\psi_{n}$ of $\mathcal N_+(\lambda)$ do not depend on $\lambda$ and we have
 $$
 \mathcal N_+(\lambda) {\bf P}f=\sum_{0\leq n\leq k R} \mu_{n}(\lambda)P_n
 $$
where  the $\mu_n$ have been defined in \eqref{eig}.

 The eigenvalues of $\mathcal N_+(0)$ on $[0,1/R]$ are the limits of
 $\mu_n(\lambda)$. As $|\lambda|\to 0+$, $\Im\lambda\geq 0$,  the
 formula \eqref{eig} and properties of the Hankel functions
(see \cite{AS}) imply that
\be\label{fe}
\mu_0(\lambda)=-\frac {1}{R\ln \lambda}\Bigl(1-\Bigl(\ln\frac {R} {2} +\frac{\pi\gamma} {2}-i\frac {\pi }{2}\Bigr)\frac{1}{\ln\lambda} +O\Bigl(\frac {1}{(\ln\lambda)^{2}}\Bigr)\Bigr),
\ee
where $\gamma$ is the Euler's constant, and
\be\label{star2}
\mu_{n}(\lambda)=|n|/(kR^2)+O(\lambda^{\epsilon}),\quad 0<|n|<k{R},
\ee
with some $\epsilon>0$.

We show the following proposition.

\begin{prop}\label{th1}
Assume $(X,|df|^2)$ has only one conical end then, for any $R$ large enough we have,
as  $|\lambda| \to 0+$, $\Im\lambda\geq 0$,
\begin{equation*}
\det_\zeta\mathcal N(\lambda)=-\frac {1}{R\ln \lambda}\det^*_\zeta\mathcal N(0)\Bigl(1-\Bigl(\ln\frac {R} {2} +\frac{\pi\gamma} {2}-i\frac {\pi }{2}\Bigr)\frac{1}{\ln\lambda} +O\Bigl(\frac {1}{(\ln\lambda)^{2}}\Bigr)\Bigr),
\end{equation*}
where $\det_\zeta\mathcal N(\lambda)$  is the zeta regularized determinant of $\mathcal N(\lambda)$,  and $\det^*_\zeta\mathcal N(0)$ is  the zeta regularized determinant of $\mathcal N(0)$ with zero eigenvalue excluded.
\end{prop}
\begin{proof}


 Due to the representation $L^2( \Sigma)= \ker P_0\oplus \ker ({\bf P}-P_0)\oplus \ker (\Id-{\bf P})$ we have
\be\label{ee}
\mathcal N(\lambda)=
\left(
\begin{array}{ccc}
\mathcal  N_{0,0} (\lambda) & \mathcal N_{0,1}(\lambda)&  \mathcal N_{0,2}(\lambda) \\
\mathcal  N_{1,0} (\lambda) & \mathcal N_{1,1}(\lambda)&  \mathcal N_{1,2}(\lambda) \\
\mathcal  N_{2,0} (\lambda) & \mathcal N_{2,1}(\lambda)&  \mathcal N_{2,2}(\lambda) \\
\end{array}
\right),
\ee
where $\mathcal N_{i,j}(\lambda)=\mathcal P_i\mathcal N(\lambda) \mathcal P_j$ with $\mathcal P_0=P_0$, $\mathcal P_1={\bf P}-P_0$, and $\mathcal P_2=\Id-{\bf P}$.

The operator $\mathcal N_{2,2}(0)$
is invertible and therefore
$
\det_\zeta \mathcal N_{2,2}(0)\neq 0
$.
Note that $\mathcal N_{2,2}(\lambda)=\mathcal P_2\mathcal N_-(\lambda)\mathcal P_2+ \mathcal P_2 \mathcal N_+(\lambda)\mathcal P_2$, where $\mathcal N_+(\lambda)\mathcal P_2=\mathcal N_+(\lambda)(\Id-{\bf P})$  is the same as in Lemma~\ref{L1} and $\mathcal N_-(\lambda)$ is a holomorphic function of $\lambda^2$ in a small neighbourhood of zero.
 This implies that
 \be\label{11}
 \det_\zeta \mathcal N_{2,2}(\lambda)- \det_\zeta  \mathcal N_{2,2}(0)=o(1), \quad|\lambda| \to 0+, \Im\lambda\geq 0.
 \ee

 Thanks to Lemma~\ref{L1} we also have
 \be\label{ps}
 \begin{aligned}
 \|\mathcal N_{2,2}(\lambda)-\mathcal N_{2,2}(0); \mathcal B(H^{1}(\Sigma),L^2(\Sigma))\|=O(\lambda^2),\\
 \|\partial_\lambda\mathcal N_{2,2}(\lambda)-\partial_\lambda\mathcal N_{2,2}(0); \mathcal B(H^{1}(\Sigma),L^2(\Sigma))\|=O(\lambda).
 \end{aligned}
 \ee

 In order to refine~\eqref{11},  we estimate the absolute value of $\partial_\lambda\ln\det_\zeta  \mathcal N_{2,2}(\lambda)$.
Since $\partial_\lambda \mathcal N_{2,2}(\lambda)$ and $\mathcal N_{2,2}^{-1}(\lambda)$ are pseudodifferential operators of order $-1$,  the operator $$\mathcal N_{2,2}^{-1}(\lambda)\partial_\lambda \mathcal N_{2,2}(\lambda)$$ is in the trace class, and hence
\begin{equation}\label{*}
\begin{aligned}
\partial_\lambda\ln\det_\zeta  \mathcal N_{2,2}(\lambda)=\Tr\bigl \{\mathcal N_{2,2}^{-1}(\lambda)\partial_\lambda \mathcal N_{2,2}(\lambda)\bigr\},
\end{aligned}\end{equation}
see~\cite{BFK, Forman}. The first estimate in~\eqref{ps} and the Neumann series for $\mathcal N^{-1}_{2,2}(\lambda)$ give
\be\label{-1}
\begin{aligned}
\mathcal N_{2,2}^{-1}(\lambda) =\bigl(\Id+L(\lambda)\bigr)\mathcal N_{2,2}^{-1}(0),\quad
\|L(\lambda); \mathcal B(H^{1}(\Sigma))\|=O(\lambda^2);\\
\mathcal N_{2,2}^{-1}(\lambda) =\mathcal N_{2,2}^{-1}(0)  \bigl(\Id+R(\lambda)\bigr),\quad
\|R(\lambda); \mathcal B(L^{2}(\Sigma))\|=O(\lambda^2).
\end{aligned}
\ee
As a consequence of~\eqref{*},~\eqref{-1}, and~\eqref{ps} we get
$$
\begin{aligned}
|\partial_\lambda & \ln\det_\zeta \mathcal N_{2,2}(\lambda)|=|\Tr\bigl \{\mathcal N_{2,2}^{-1}(\lambda)\partial_\lambda \mathcal N_{2,2}(\lambda)\bigr\}|\\
& \leq \|  \mathcal N_{2,2}(\lambda) (\Id+L(\lambda))\mathcal  N_{2,2}^{-2}(0)(\Id+R(\lambda))\partial_\lambda \mathcal N_{2,2}(\lambda) \|_1\\
& \leq\|\mathcal N_{2,2}(\lambda);\mathcal B(H^{1}(\Sigma),L^2(\Sigma))\|  \|\Id+L(\lambda); \mathcal B(H^{1}(\Sigma))\|\|\mathcal N_{2,2}^{-2}(0)\|_1
 \\& \times\|\Id-R(\lambda); \mathcal B(L^2(\Sigma))\|  \|\partial_\lambda \mathcal N_{2,2}(\lambda); \mathcal B(H^{1}(\Sigma),L^2(\Sigma))\| =O(1),
\end{aligned}
$$
where $\|\cdot\|_1$ is the trace norm. This together with~\eqref{11} implies $|\partial_\lambda \det_\zeta  \mathcal N_{2,2}(\lambda) |=O(1)$. Now, as a refinement of~\eqref{11}, we obtain
$$
\det_\zeta \mathcal N_{2,2}(\lambda)-\det_\zeta \mathcal N_{2,2}(0)
=O(\lambda).
$$
This together with~\eqref{ee} implies
\be\label{12}
\begin{aligned}
\det_\zeta  &  \mathcal N(\lambda)=\detF\left(
\begin{array}{ccc}
 \mathcal N_{0,0}(\lambda) & \mathcal N_{0,1}(\lambda) &   \mathcal N_{0,2}(\lambda)  \mathcal N_{2,2}(\lambda)^{-1}\\
 \mathcal N_{1,0}(\lambda) & \mathcal N_{1,1}(\lambda) &   \mathcal N_{1,2}(\lambda)  \mathcal N_{2,2}(\lambda)^{-1}\\
 \mathcal N_{2,0}(\lambda) & \mathcal N_{2,1}(\lambda) &  \Id
 \end{array}
\right)\det_\zeta\left(
\begin{array}{ccc}
 \Id & 0 &   0 \\
  0&  \Id&0 \\
  0& 0 & \mathcal N_{2,2}(\lambda)
\end{array}
\right)
\\
&=\detF\left(
\begin{array}{ccc}
 \mathcal N_{0,0}(\lambda) & \mathcal N_{0,1}(\lambda) &   \mathcal N_{0,2}(\lambda)  \mathcal N_{2,2}(\lambda)^{-1}\\
 \mathcal N_{1,0}(\lambda) & \mathcal N_{1,1}(\lambda) &   \mathcal N_{1,2}(\lambda)  \mathcal N_{2,2}(\lambda)^{-1}\\
 \mathcal N_{2,0}(\lambda) & \mathcal N_{2,1}(\lambda) &  \Id
 \end{array}
\right)\bigl(\det_\zeta \mathcal N_{2,2}(0)\bigr)
\bigl(1+O(\lambda)
\bigr);
\end{aligned}
\ee
see~\cite{KV} for the first equality. On the next step we rely on the estimate
\be\label{13}
|\detF(\Id +A)-\detF(\Id+B)|\leq \|A-B\|_1 e^{\|A\|_1+\|B\|_1+1},
\ee
see~\cite[f-la (3.7), and references therein]{Simon}, for
$$
\Id+A=\left(
\begin{array}{ccc}
 \mathcal N_{0,0}(\lambda) & 0 &  0\\
 0 & \mathcal N_{1,1}(0) &  \mathcal N_{1,2}(0)\mathcal N_{2,2}^{-1}(0)\\
0 &\mathcal N_{2,1}(0) &  \Id
 \end{array}
\right),
$$
$$
 \Id+B=\left(\begin{array}{ccc}
 \mathcal N_{0,0}(\lambda) & \mathcal N_{0,1}(\lambda) &   \mathcal N_{0,2}(\lambda)  \mathcal N_{2,2}(\lambda)^{-1}\\
 \mathcal N_{1,0}(\lambda) & \mathcal N_{1,1}(\lambda) &   \mathcal N_{1,2}(\lambda)  \mathcal N_{2,2}(\lambda)^{-1}\\
 \mathcal N_{2,0}(\lambda) & \mathcal N_{2,1}(\lambda) &  \Id
 \end{array}
\right).
$$
Since $\mathcal N_-(0)$ is a selfadjoint operator in $L^2(\Sigma)$ and $\ker \mathcal N_-(0)=\{c\in\Bbb C\}$, we have $\mathcal N_-(0)P_0=P_0\mathcal N_-(0)=0$. Then thanks to
$$P_i\bigl(\mathcal N(\lambda)-\mathcal N(0)\bigr)P_j=\delta_{ij }\bigl(\mu_j(\lambda)-\mu_j(0)\bigr)+P_i\bigl(\mathcal N_-(\lambda)-\mathcal N_-(0)\bigr)P_j,\quad i,j\in[0,1/R],$$
where $\delta_{ij }$ is the Kronecker delta function, and
$$P_j\bigl(\mathcal N(\lambda)-\mathcal N(0)\bigr)(\Id-{\bf P})=P_j\bigl(\mathcal N_-(\lambda)-\mathcal N_-(0)\bigr)(\Id-{\bf P}),\quad j\in[0,1/R],$$
together with~\eqref{star2} and~\eqref{-1}, we obtain $\|A-B\|_1=O(\lambda^\epsilon)$ with some $\epsilon>0$.
From~\eqref{12} and~\eqref{13} we get
\be\label{lll}
\begin{aligned}
\det_\zeta    \mathcal N(\lambda)=\detF\left(
\begin{array}{ccc}
 \mathcal N_{0,0}(\lambda) & 0&   0\\
 0 & \mathcal N_{1,1}(0) &   \mathcal N_{1,2}(0)  \mathcal N_{2,2}(0)^{-1}\\
 0 & \mathcal N_{2,1}(0) &  \Id
 \end{array}
\right)\det_\zeta \mathcal N_{2,2}(0)
(1+O(\lambda^\epsilon)).
\end{aligned}
\ee
It remains to note that
$$
\mathcal N_{0,0}(\lambda)= (\mu_0(\lambda)+P_0(\mathcal N_-(\lambda)-\mathcal N_-(0))P_0=(\mu_0(\lambda)+O(\lambda^2))P_0,
$$
$$\det_\zeta^*\mathcal N(0)= \detF\left(
\begin{array}{cc}
 \mathcal N_{1,1}(0) &   \mathcal N_{1,2}(0)  \mathcal N_{2,2}(0)^{-1}\\
 \mathcal N_{2,1}(0) &  \Id
 \end{array}
\right)\det_\zeta \mathcal N_{2,2}(0).$$
This together with~\eqref{lll} and~\eqref{fe} completes the proof.
\end{proof}

\begin{cor}\label{C1} The spectral shift function $\xi$ in~\eqref{Krein} satisfies
\be\label{cc1}
\xi(\lambda\,;\,\Delta,\Delta^D)=( \ln\lambda^2)^{-1} +O\bigl((\ln\lambda)^{-2}\bigr),\quad \lambda \to 0+.
\ee
\end{cor}
\begin{proof} By~\cite[Theorem 3.5]{Carron1}  we have $\xi(\lambda)=\pi^{-1} \operatorname{Arg}\det \mathcal N (\sqrt{\lambda^2+i0})$ as $\lambda^2\to 0+$, where $\operatorname{Arg}z\in (-\pi,\pi]$, and  $\xi(\lambda)=0$ if  $\lambda^2<0$. Calculation of the argument in the asymptotic obtained in Propositon~\ref{th1} gives~\eqref{cc1}.
\end{proof}

\subsubsection{The case $K>1$}
 \label{R2} Let us outline the  changes in Proposition~\ref{th1} and Corollary~\ref{C1} needed in the case  $K>1.$
Now we have  $\mathcal N_+(\lambda)=\oplus_{j=1}^K \mathcal N_+^{(j)}(\lambda)$, where each  $\mathcal N_+^{(j)}(\lambda)$ is defined on the circle $\{y\in\Bbb C_j: |y|=R^{1/k_j}\}$ as in~\eqref{Next}. The first eigenvalue of $\mathcal N_+^{(j)}(\lambda)$ is $\mu_0(\lambda)$ and the corresponding eigenspace consists of constant functions on the circle.
As a consequence, in the estimate \eqref{lll} the eigenvectors in $\Ncal_+$  with
eigenvalue $\mu_0(\lambda)$ contribute at the order $O(\frac{1}{\ln \,\lambda})$ instead of $O(\lambda^\epsilon)$ and this is not good enough for our purpose.

We thus introduce $\bf P_0$ the orthogonal projection onto the eigenspace of $\mathcal N_+(\lambda)$ corresponding to $\mu_0(\lambda)$   (note that  $\bf P_0$ does not depend on $\lambda$,  and that $\operatorname{rank} {\bf P_0}=K$). Observe that we have $\ker \mathcal N(0)\subset \ker(\Id-\bf{P_0})$.  We repeat the argument of Proposition~\ref{th1}, where $\mathcal P_0$ is now the orthogonal projection onto $\ker\mathcal N(0)=\{c\in\Bbb C\}$, $\mathcal P_1=(\Id-\mathcal P_0 ) {\bf P}$ where $\bf P$ is the spectral projection of $\mathcal N_+(0)$ on the interval $[0,1/R]$, and $\mathcal P_2=(\Id-{\bf P})$. Clearly, $\mathcal N_+(\lambda)\mathcal P_0=\mu_0(\lambda)\mathcal P_0$ and $\mathcal N_+(\lambda)(\Id-\mathcal P_0)=(\Id-\mathcal P_0)\mathcal N_+(\lambda)$.

 The same argument as in the case $K=1$ leads to
$$
\begin{aligned}
\det_\zeta    \mathcal N(\lambda)=-\frac{1}{R\ln\lambda}\detF\left(
\begin{array}{cc}
  \mathcal N_{1,1}(0)+\mu_0(\lambda){\bf P_0}(\Id-\mathcal P_0) &   \mathcal N_{1,2}(0)  \mathcal N_{2,2}(0)^{-1}\\
 \mathcal N_{2,1}(0) &  \Id
 \end{array}
\right)
\\
\times \det_\zeta \mathcal N_{2,2}(0)\Bigl(1-\Bigl(\ln\frac {R} {2} +\frac{\pi\gamma} {2}-i\frac {\pi }{2}\Bigr)\frac{1}{\ln\lambda} +O\Bigl(\frac {1}{(\ln\lambda)^{2}}\Bigr)\Bigr).
\end{aligned}
$$
(Note that in the case $K=1$ we have $\bf P_0=\mathcal P_0$ and the term  $\mu_0(\lambda){\bf P_0}(\Id-\mathcal P_0)$ does not appear.)
This together with~\eqref{fe} and~\eqref{13}  gives
$$
\begin{aligned}
\det_\zeta    \mathcal N(\lambda)=-\frac{1}{R\ln\lambda}\detF\left(
\begin{array}{cc}
  \mathcal N_{1,1}(0)-\frac 1 {R\ln \lambda}{\bf P_0}(\Id-\mathcal P_0) &   \mathcal N_{1,2}(0)  \mathcal N_{2,2}(0)^{-1}\\
 \mathcal N_{2,1}(0) &  \Id
 \end{array}
\right)
\\
\times \det_\zeta \mathcal N_{2,2}(0)\Bigl(1-\Bigl(\ln\frac {R} {2} +\frac{\pi\gamma} {2}-i\frac {\pi }{2}\Bigr)\frac{1}{\ln\lambda} +O\Bigl(\frac {1}{(\ln\lambda)^{2}}\Bigr)\Bigr).
\end{aligned}
$$
Here the Fredholm determinant is a holomorphic function of the parameter $\tau:=\frac{1}{R\ln\lambda}$, we have
\be\label{LR}
\det_\zeta    \mathcal N(\lambda)=-\frac {1}{R\ln \lambda}\det^*_\zeta\mathcal N(0)\Bigl(1-\Bigl(C+\ln\frac {R} {2} +\frac{\pi\gamma} {2}-i\frac {\pi }{2}\Bigr)\frac{1}{\ln\lambda} +O\Bigl(\frac {1}{(\ln\lambda)^{2}}\Bigr)\Bigr)
\ee
with some constant $C=C(R)$.

We observe that $C$ must be real since $\det_\zeta    \mathcal N(\lambda)$ is positive for $\lambda\in\Bbb C$, $ \operatorname{Arg}\lambda=\pi/2$ (as  $\mathcal N(\lambda)$ is a positive self-adjoint operator for those values of $\lambda$).  Thus $C$ does not influence the calculation of the argument in the asymptotic of $\det_\zeta    \mathcal N(\lambda)$ and Corollary~\ref{C1} remains valid for $K>1$.
We will use Corollary~\ref{C1}  to define a relative determinant of $(\Delta,\mathring{\Delta})$ at $\lambda=0$.

\subsection{The relative determinant and the gluing formula at  $\lambda=0$}
In this section we prove the following Theorem.

\begin{thm} \label{M}The gluing formula
$$
\det_\zeta (\Delta,\mathring\Delta)= C \det^*_\zeta \mathcal N\cdot\det_\zeta\Delta^D_-
$$
is valid, where $\mathcal N$, $\Delta_-^D$ depend on $R$. The constant $C$ depends on $R$ but not on the moduli parameters
$z_1,\dots z_M.$
\end{thm}

Observe that this theorem first requires a definition for the left-hand side of the equality.
Once this is done, we will let $\lambda$ go to zero in \eqref{sg1} and study the limit of both sides.

As before the case $K =1$ is simpler than the general one. We will present the proof for this case first.
The case $K>1$ is more technically involved but the arguments we need can be adapted from \cite{Hassell-Zelditch}.

\subsubsection{ The case $K=1$}
In this case, the definition of $\det_\zeta
(\Delta,\mathring\Delta)$ is rather straightforward, since, for $K=1$, a conventional regularization (see, e.g.,~\cite{HKK} and references therein) for the relative zeta function makes sense. Indeed, the first integral in   the  representation
\begin{equation}\label{dd1}
\zeta(s; \Delta,\mathring \Delta)=\left(\int_0^1+\int_1^\infty\right)\frac{t^{s-1}}{\Gamma(s)} \Tr \bigl(e^{-t\Delta}-e^{-t\mathring\Delta}\bigr)\,dt
\end{equation}
 defines an analytic in $\Re s>1$ function that has a meromorphic continuation to $\Bbb C$ with no pole at zero by the usual argument based on short time heat trace asymptotic
\begin{equation}\label{ST}
\Tr \bigl(e^{-t\Delta }- e^{-t\mathring{\Delta}}\bigr)\sim \sum_{j\geq- 2} a_j t^{j/2},\quad t\to 0+.
\end{equation}

For the second integral we need the long time heat trace behaviour given by the following lemma.

\begin{lem} \label{l++} Assume that $K=1$. Then
$$
\Tr \bigl(e^{-t\Delta }- e^{-t\mathring\Delta}\bigr)=O\bigl((\ln t)^{-2}\bigr)\text{ as } t\to+\infty.
$$
\end{lem}
\begin{proof}  Since $\Delta^D_+\equiv \mathring{\Delta}^D_+$  and the operators  $\Delta_-^D$, $\mathring{\Delta}^D_-$ are Dirichlet Laplacians on compact manifolds, we have
\begin{equation}\label{123}
\begin{split}
\Tr & \bigl(e^{-t\Delta } - e^{-t\mathring\Delta}\bigr)\\ & =\,\Tr \bigl(e^{-t\Delta }- e^{-t\Delta^D_-\oplus\Delta^D_+}\bigr)-\Tr \bigl(e^{-t\mathring{\Delta} }- e^{-t\mathring{\Delta}^D_-\oplus\mathring{\Delta}^D_+}\bigr) +\Tr e^{-t\Delta^D_-}-\Tr e^{-t\mathring{\Delta}^D_-}\\
&=\,-t\int_0^\infty e^{-t\lambda^2} (\xi(\lambda\,;\,\Delta,\Delta^D)-\xi(\lambda\,;\,\mathring{\Delta},\mathring{\Delta}^D))2\lambda\,d\lambda +O(e^{-t\delta}),\quad t\to+\infty,
\end{split}
\end{equation}
where  $\delta>0$ is the smallest eigenvalue  in the spectra of $\Delta^D_-$ and $\mathring{\Delta}^D_-$. (In~\eqref{123}  we also used~\eqref{Krein} for $\Delta$ and $\mathring{\Delta}$.)
As a consequence of Corollary~\ref{C1} (which is also valid for $\mathring\xi$ in the case $K=1$) we have
$$
\xi(\lambda\,;\,\Delta,\Delta^D)-\xi(\lambda,\mathring{\Delta},\mathring{\Delta}^D)=O\bigl((\ln\lambda)^{-2}\bigr),\quad\lambda\to 0+.$$
This together with~\eqref{123} implies the assertion; see e.g.~\cite[Theorem 1.7]{Fedoryuk} for details.
\end{proof}

As a consequence, the second integral in~\eqref{dd1} defines a holomorphic in $\Re s<0 $ function that has a  continuous in $\Re s\leq 0$ derivative. Thus $\zeta(s; \Delta,\mathring \Delta)$ is a meromorphic function in $\Re s<0$ and $\zeta'(s; \Delta,\mathring \Delta)$ tends to a certain limit $\zeta'(0; \Delta,\mathring \Delta)$ as $s\to0-$. The relative zeta regularized determinant is defined to be
\begin{equation}\label{dd2}
\det_\zeta(\Delta,\mathring\Delta)=e^{-\zeta'(0; \Delta,\mathring \Delta)}.
\end{equation}

We now prove the gluing formula in the case $K=1$. First observe that by Proposition~\ref{th1} (applied also to $\mathring{\mathcal N}(\lambda)$) we have
\begin{equation}\label{NN}
\frac{\det_\zeta\mathcal N(\lambda)}{\det_\zeta\mathring{\mathcal N}(\lambda)}\to \frac{\det_\zeta^*\mathcal N(0)}{\det_\zeta^*\mathring{\mathcal N}(0)} \text{ as } \lambda^2\to 0-,\  K=1.
\end{equation}
The limit $\lambda \rightarrow 0$ is then addressed by the
\begin{prop} \label{fff}In the case $K=1$ we have
$$
\det_\zeta(\Delta-\lambda^2,\mathring\Delta-\lambda^2)\to\det_\zeta(\Delta,\mathring\Delta)\text{ as }\lambda^2\to 0-,
$$
where the determinant  $\det_\zeta(\Delta,\mathring\Delta)$ has been defined in~\eqref{dd2}.
\end{prop}

\begin{proof}
Let us write the relative zeta function in the form
$$\begin{aligned}
\zeta(s;\Delta-\lambda^2,\mathring\Delta -\lambda^2)=\left(\int_0^1+\int_1^\infty \right)\frac{t^{s-1} e^{t \lambda^2 } }{\Gamma(s)}\Tr \bigl(e^{-t\Delta }- e^{-t\mathring\Delta}\bigr)\,dt.
\end{aligned}
$$
Thanks to~\eqref{t0} the first integral converges  for $\Re s>1$ uniformly in $\lambda\leq 0$ and has a meromorphic continuation to $\Bbb C$ with no pole at zero (by the usual argument based on the short time heat trace asymptotic~\eqref{ST}). Due to Lemma~\ref{l++} the second integral  defines a holomorphic in $\Re s<0$ and continuous in $\lambda^2\leq 0$ and $\Re s\leq 0$ function. Moreover, as $1/\Gamma(s)$ has a first order zero at $s=0$, Lemma~\ref{l++} also implies that the first derivative with respect to $s$ of the second integral  is also continuous in  $\lambda^2\leq 0$ and $\Re s\leq 0$.  Thus we obtain
$$
\zeta' (0;\Delta-\lambda^2,\mathring\Delta-\lambda^2)\to  \zeta' (0;\Delta,\mathring\Delta), \quad\lambda^2\to 0-.
$$
where $ \zeta' (0;\Delta,\mathring\Delta)$ is defined using \eqref{dd1}.
\end{proof}

\begin{proof}[Proof of Theorem~\ref{M} in the case $K=1$] We pass to the limit as $\lambda^2\to 0-$ in~\eqref{sg1}.  Since $\Delta^D_-$
is positive, we have $\det_\zeta(\Delta^D_--\lambda^2)\to \det_\zeta \Delta^D_-$ as $\lambda^2\to 0-$, and the same is true for $\mathring\Delta^D_-$.
Thanks to~\eqref{NN} and Proposition~\ref{fff} we obtain
$$
\frac{\det_\zeta(\Delta,\mathring\Delta)\det_\zeta \mathring\Delta^D_-}{\det_\zeta \Delta_-^D}=\frac{\det_\zeta^*\mathcal N(0)}{\det_\zeta^*\mathring{\mathcal N}(0)},
$$
which proves~Theorem~\ref{M}, where $\mathcal N(0)$ is denoted by $\mathcal N$ and the constant
$$
C= \bigl(\det_\zeta \mathring\Delta^D_-    \det_\zeta^*\mathring{\mathcal N}(0)\bigr)^{-1}
$$
is moduli independent.
\end{proof}

\subsubsection{The case $K >1$}

In the case $K>1$ we have $\mathring{\mathcal N}(\lambda)=\oplus_{j=1}^K \mathring{\mathcal N}^{(j)}(\lambda)$, where $\mathring{\mathcal N}^{(j)}(\lambda)$ is the Neumann jump operator on the circle $\{y\in\Bbb C_j:|y|=  R^{1/k_j}\}$ located on the infinite cone $(\Bbb C_j, |d y^{k_j}|^2)$. We have
$$
\det_\zeta  \mathring{\mathcal N}(\lambda)=\prod_{j=1}^K \det_\zeta  \mathring{\mathcal N}^{(j)}(\lambda),\quad \det_\zeta^*  \mathring{\mathcal N}(0)=\prod_{j=1}^K \det_\zeta^*  \mathring{\mathcal N}^{(j)}(0).
$$
We apply Proposition~\ref{th1} to each  $\det_\zeta  \mathring{\mathcal N}^{(j)}(\lambda)$, $j=1,\dots,K$ and get
\be\label{dev3+}
\det_\zeta  \mathring{\mathcal N}(\lambda)=
(-R\ln \lambda)^{-K}\det^*_\zeta\mathcal N(0)\Bigl(1-\Bigl(\ln\frac {R} {2} +\frac{\pi\gamma} {2}-i\frac {\pi }{2}\Bigr)\frac{K}{\ln\lambda} +O\Bigl(\frac {1}{(\ln\lambda)^{2}}\Bigr)\Bigr),
\ee
as $|\lambda| \to 0+$, $\Im\lambda\geq 0$.

Thanks to the relation $\xi(\lambda\,;\,\mathring{\Delta},\mathring{\Delta}^D)=\pi^{-1} \operatorname{Arg}\det \mathring{\mathcal N} (\sqrt{\lambda^2+i0})$ as $\lambda^2\to 0+$,  calculation of the argument in~\eqref{dev3+} leads to
$$
\xi(\lambda\,;\,\mathring{\Delta},\mathring{\Delta}^D)\,=\,K(\ln\lambda^2)^{-1}+O((\ln \lambda)^{-2}),\quad \lambda\to 0+,
$$
where $\xi(\cdot \,;\,\mathring{\Delta},\mathring{\Delta}^D)\in L^1(\Bbb R_+, (1+\lambda^2)^{-2}\,d\lambda^2)$  is the spectral shift function satisfying
$$
\Tr\bigl( (\mathring\Delta+1)^{-1}-(\mathring\Delta^D+1)^{-1}
\bigr)=-\int_0^\infty \xi(\lambda\,;\,\mathring{\Delta},\mathring{\Delta}^D)(1+\lambda^2)^{-2}\,d\lambda^2;
$$
cf.~Corollary~\ref{C1}.
This together with  Corollary~\ref{C1} gives
\be\label{dev2}
\xi(\lambda,\Delta,\Delta^D)-\xi(\lambda,\mathring{\Delta},\mathring{\Delta}^D) =-(K-1)(\ln\lambda^2)^{-1}+O((\ln \lambda)^{-2}),\quad \lambda\to 0+.
\ee
Besides, Proposition~\ref{th1} together with~\eqref{dev3+} implies that
\be\label{NNK}
\left(\ln\frac{i}{\lambda}\right)^{1-K}\frac{\det_\zeta\mathcal N(\lambda)}{\det_\zeta\mathring{\mathcal N}(\lambda)}\to R^{K-1}\frac{\det_\zeta^*\mathcal N(0)}{\det_\zeta^*\mathring{\mathcal N}(0)} \text{ as } \lambda^2\to 0-,\ K\geq 1.
\ee

Recall that for $\lambda^2<0$ the relative zeta function is defined as the meromorphic  continuation of
\begin{equation}\label{eq:defzet}
\zeta(s;\Delta-\lambda^2,\mathring\Delta -\lambda^2)=\int_0^\infty\frac{t^{s-1} e^{t \lambda^2 } }{\Gamma(s)}\Tr \bigl(e^{-t\Delta }- e^{-t\mathring\Delta}\bigr)\,dt
\end{equation}
from $\Re s >1$.

We have
$$
\Tr \bigl(e^{-t\Delta }- e^{-t\mathring\Delta}\bigr)=\Tr \bigl(e^{-t\Delta }-e^{-t\Delta^D_-\oplus\Delta_+^D}\bigr)-\Tr \bigl( e^{-t\mathring\Delta}-e^{-t\mathring\Delta^D_-\oplus\mathring\Delta^D_+}\bigr)+\Tr e^{-t\Delta^D_-}-\Tr e^{-t\mathring{\Delta}^D_-}.
$$
(Now the short time asymptotic~\eqref{ST}  is a consequence of~\eqref{t0} and similar short time asymptotics for $ \Tr \bigl(e^{-t\mathring{\Delta} }- e^{-t\mathring{\Delta}^D}\bigr)$, $\Tr e^{-t\Delta_-^D}$,  and $\Tr e^{-t\mathring{\Delta}_-^D}$.)
Let $$N(\lambda)=\sum_{j: \lambda_j^2\leq \lambda^2} \dim\ker (\Delta_-^D-\lambda_j^2),$$ where $\lambda_j^2$ are the eigenvalues of $\Delta_-^D$,  be the counting function of $\Delta_-^D$.
Similarly, let $\mathring{N}(\lambda)$ be the counting function of $\mathring{\Delta}_-^D$. Then
$$\Tr (e^{-t\Delta^D_-}-e^{-t\mathring{\Delta}^D_-})=t\int_0^\infty e^{-t\lambda^2 }\Bigl(N(\lambda)-\mathring{N}(\lambda)\Bigr)2\lambda\,d\lambda$$
and
$$
\xi(\lambda; \Delta\,;\,\mathring{\Delta})=\xi(\lambda\,;\,\Delta,\Delta^D)-\xi(\lambda\,;\,\mathring{\Delta},\mathring{\Delta}^D)-N(\lambda)+\mathring N(\lambda)
$$
is the spectral shift function for the pair $(\Delta,\mathring{\Delta})$ such that
$$
\Tr \bigl(e^{-t\Delta }- e^{-t\mathring\Delta}\bigr)=-t\int_0^\infty e^{-t\lambda^2}\xi(\lambda; \Delta,\mathring{\Delta})2\lambda\,d\lambda.
$$
Since the operators $\Delta_-^D$ and $\mathring{\Delta}_-^D$ are positive, from~\eqref{dev2} it follows the asymptotic
\be\label{devxias}
\xi(\lambda; \Delta,\mathring{\Delta})=-(K-1)(\ln\lambda^2)^{-1}+O((\ln \lambda)^{-2}),\quad \lambda\to 0+.
\ee

 Introduce a cutoff function $\chi\in C^\infty(\Bbb R)$ such that $\chi(\mu)=1$ for $\mu<1/2$ and $\chi(\mu)=0$ for $\mu>3/4$.
Following the scheme in~\cite{Hassell-Zelditch} we write
$$
\Tr \bigl(e^{-t\Delta }- e^{-t\mathring\Delta}\bigr)=e_1(t)+e_2(t),
$$
where
$$
e_1(t)=-t\int_0^\infty e^{-t\mu^2}\chi(\mu)\xi(\mu; \Delta,\mathring{\Delta})2\mu\,d\mu,
$$
$$
e_2(t)=-t\int_0^\infty e^{-t\mu^2}(1-\chi(\mu)) \xi(\mu; \Delta,\mathring{\Delta})2\mu\,d\mu;
$$
cf.~\eqref{Krein}. Note that  $e_2$ is exponentially decreasing as $t\to+\infty$. Thanks to the short time asymptotic~\eqref{ST} and smoothness of $e_1$ at $t=0$, we see that $e_2(t)$ has a short time asymptotic of the same form. Therefore for $\lambda^2\leq 0$ the holomorphic in $\Re s>1 $ zeta function
$$
\zeta_2(s; \lambda^2)= \int_0^\infty\frac{t^{s-1} e^{t \lambda^2 } }{\Gamma(s)} e_2(t)\,dt
$$
continues as a meromorphic function to  $\Bbb C$ with no pole at zero and  $ \zeta_2'(0,\lambda^2)\to \zeta_2'(0,0)$ as $\lambda^2\to0-$ by the usual argument.

We are now in position to define the regularized determinant $\det_\zeta(\Delta,\mathring{\Delta})$
We start from
$
\zeta(s;\Delta,\mathring\Delta) =\zeta_1(s;0)+\zeta_2(s;0),
$
where $\zeta_1(s;0)$ is defined by
\be\label{devlt0}
\zeta_1(s; 0)= \int_0^\infty\frac{t^{s-1}}{\Gamma(s)} e_1(t)\,dt=s\int_0^\infty (-\mu^2)^{-s-1}\chi(\mu)\xi(\lambda; \Delta,\mathring{\Delta})2\mu\,d\mu.
\ee
 From this expression and~\eqref{dev2} one can easily see that $\zeta_1(s;0)$ is a holomorphic function in $\Re s<0$, and we already know that
$\zeta_2(s;0)$is a meromorphic function of $s\in \Bbb C$ with no pole at $0$.

The asymptotic behaviour of $\zeta(s;0)$ near $s=0$ is given by the following proposition.
\begin{prop}\label{devp2}
Set
\begin{equation*}
\begin{split}
Z'_0:=& \zeta_2'(0;0)+(K-1)\lim_{\delta\to0+}\left(\int_\delta^\infty \frac{\chi(\mu)d\mu}{\mu\ln\mu}+\ln\ln\frac{1}{\delta}\right)\\
&- 2\int_0^\infty \mu^{-1}\chi(\mu)  \bigl(\xi(\mu; \Delta,\mathring{\Delta})-(K-1)(\ln\mu^2)^{-1} \bigr)\,d\mu+(K-1)(\gamma+\ln 2).
\end{split}
\end{equation*}
When $s\to 0-$ we have
$$
\zeta(s;\Delta,\mathring\Delta)=  \zeta_2(0;0)+s(K-1)\ln (-s) +s Z'_0 +o(s).
$$
\end{prop}

This proposition gives way to the following definition of the regularized relative determinant of $\Delta$ in case $K>1$.

\begin{definition}
 $$\det_\zeta (\Delta, \mathring{\Delta}):= e^{-Z'_0}\,.$$
\end{definition}

\begin{proof}  Since $\zeta_2(s;0)$ is a meromorphic function of $s$ with no pole at zero, as $s\to 0-$
we have
$$
\zeta(s;\Delta,\mathring\Delta)=\zeta_1(s;0)+\zeta_2(0;0)+s\zeta_2'(0;0)+O(s^2).
$$
It remains to study the behaviour of
$$
\begin{aligned}
\zeta_1(s;0)=s\int_0^\infty (-\mu^2)^{-s-1}\chi(\mu)\xi(\mu; \Delta,\mathring{\Delta})2\mu\,d\mu
\end{aligned}
$$
as $s\to0-$. We represent the last integral as a sum of two integrals.
Due to~\eqref{devxias} the first integral
$$
s\int_0^\infty (-\mu^2)^{-s-1}\chi(\mu)  \bigl(\xi(\mu; \Delta,\mathring{\Delta})+(K-1)(\ln\mu^2)^{-1} \bigr)2\mu\,d\mu
$$
converges uniformly in $s\leq 0$ and thus gives the contribution
$$
-s\int_0^\infty 2\mu^{-1}\chi(\mu)  \bigl(\xi(\mu; \Delta,\mathring{\Delta})+(K-1)(\ln\mu^2)^{-1} \bigr)\,d\mu
$$
into the expansion of $\zeta(s;\Delta,\mathring\Delta)$. For the second integral we have
\begin{gather*}
s(1-K)\int_0^\infty (-\mu^2)^{-s-1}\chi(\mu) (\ln\mu^2)^{-1}
2\mu\,d\mu = \\
s(1-K)\left(-\ln(-s)+\gamma+\ln 2 -\lim_{\delta\to0+}\left(\int_\delta^\infty \frac{\chi(\mu)d\mu}{\mu\ln\mu}+\ln\ln\frac{1}{\delta}\right)+o(1)\right);
\end{gather*}
see~\cite[p.13]{Hassell-Zelditch}.
\end{proof}

The proof of the gluing formula will also require that we understand the limit when $\lambda$ goes to $0.$
At this stage we have
$$
\zeta(s;\Delta-\lambda^2,\mathring\Delta -\lambda^2)=\zeta_1(s;\lambda^2)+\zeta_2(s;\lambda^2),
$$
where only properties of the zeta function
\be\label{devlt}
\zeta_1(s; \lambda^2)= \int_0^\infty\frac{t^{s-1} e^{t \lambda^2 } }{\Gamma(s)} e_1(t)\,dt=s\int_0^\infty (\lambda^2-\mu^2)^{-s-1}\chi(\mu)\xi(\lambda; \Delta,\mathring{\Delta})2\mu\,d\mu
\ee
remain unknown. Notice that  the last integrand is compactly supported and therefore the integral converges uniformly near $s=0$ for fixed $\lambda^2<0$ due to~\eqref{devxias}. We get
\be\label{devstst}
 \zeta_1'(0; \lambda^2)=\int_0^\infty (\lambda^2-\mu^2)^{-1}\chi(\mu)\xi(\mu; \Delta,\mathring{\Delta})2\mu\,d\mu.
\ee

\begin{prop}\label{devp1}
As $\lambda^2\to 0-$ we have
$$
\begin{aligned}
\zeta'(0;\Delta-\lambda^2,\mathring\Delta -\lambda^2)=  \ln\left(\ln\frac i \lambda\right)^{1-K} +(K-1)\lim_{\delta\to0+}\left(\int_\delta^\infty\frac{\chi(\mu)\,d\mu}{\mu\ln\mu}+\ln\ln\frac{1}{\delta}\right)\\
-2\int_0^\infty\mu^{-1}\chi(\mu)\Bigl(\xi(\mu; \Delta,\mathring{\Delta})+(K-1)(\ln\mu^2)^{-1}\Bigr)\,d\mu +\zeta_2'(0;0)+o(1).
\end{aligned}
$$
\end{prop}
\begin{proof} We only need to study the behaviour of $\zeta'(0;\lambda^2)$ in~\eqref{devstst} as $\lambda^2\to0-$.
Thanks to~\eqref{devxias} the integral
$$
\int_0^\infty (\lambda^2-\mu^2)^{-1}\chi(\mu)\Bigl(\xi(\mu; \Delta,\mathring{\Delta})+(K-1)(\ln\mu^2)^{-1} \Bigr)2\mu\,d\mu
$$
converges uniformly in $\lambda^2\leq 0$ and thus tends to
$$
-\int_0^\infty 2 \mu^{-1}\chi(\mu)\Bigl(\xi(\mu; \Delta,\mathring{\Delta})+(K-1)(\ln\mu^2)^{-1} \Bigr)\,d\mu
$$
as $\lambda^2\to0-$. It remains to note that
$$
\int_0^\infty (\lambda^2-\mu^2)^{-1}\chi(\mu)(\ln\mu^2)^{-1}2\mu\,d\mu=\ln \ln\frac{i}{\lambda}-\lim_{\delta\to 0+}\left(\int_\delta^\infty \frac{\chi(\mu)\,d\mu}{\mu\ln\mu}+\ln\ln \frac 1 \delta  \right)+o(1);
$$
as $\lambda^2\to 0-$; see~\cite[p. 12 and appendix]{Hassell-Zelditch}.
\end{proof}

Note that  by definition of $\det_\zeta (\Delta,\mathring{\Delta})=e^{-Z'_0}$ we have
 $$\zeta'(0;\Delta-\lambda^2,\mathring\Delta -\lambda^2)=  \ln\left(\ln\frac i \lambda\right)^{1-K} -\bigl(-Z'_0-(K-1)(\gamma+\ln 2)\bigr) +o(1),\quad \lambda^2\to0-,$$ see Prop.~\ref{devp1}.

\begin{proof}[Proof of Theorem~\ref{M} in the general case]
From Propositions~\ref{devp1} and~\ref{devp2} we immediately get
\be\label{DD}
\left(\ln\frac{i}{\lambda}\right)^{1-K}\det_\zeta(\Delta-\lambda^2;\mathring{\Delta}-\lambda^2)\to e^{(1-K)(\gamma+\ln 2)}\det_\zeta^*(\Delta,\mathring{\Delta}),\quad\lambda^2\to0-.
\ee
We  pass in~\eqref{sg1} to the limit as $\lambda^2\to0-$. Taking into account~\eqref{NNK} and~\eqref{DD} we obtain
$$
\frac{\det_\zeta(\Delta,\mathring{\Delta})\det_\zeta\mathring\Delta^D_-}{\det_\zeta\Delta_-^D}=\bigl(Re^{\gamma+\ln 2}\bigr)^{K-1}\frac{\det_\zeta^*\mathcal N(0)}{\det_\zeta^*\mathring{\mathcal N}(0)}.
$$
This proves Theorem~\ref{M}, where $\mathcal N\equiv\mathcal N(0)$ and the constant $$C=\bigl(R e^{\gamma+\ln 2}\bigr)^{K-1}/\bigl(\det_\zeta\mathring\Delta^D_- \det_\zeta^*\mathring{\mathcal N}(0)\bigr)$$
is moduli independent.
\end{proof}

\begin{remark}\label{rk:mgm}
The proof of the gluing formula holds verbatim for a more general class of metrics under the following two assumptions.
First, the structure at infinity should be given by a finite union of conical/Euclidean ends. Second, we have to assume that nothing bad happens with
the Laplace operator $\Delta_-^D$ of the compact part. In particular,  it should have a well-defined zeta-function that extends
to the complex plane with no pole at $0$. This works for instance if the metric is smooth in the compact part or, if it is flat with conical singularities
and the Friedrichs extension is chosen.
\end{remark}
\subsection  {Closing the Euclidean (conical) ends with the help of gluing formulas}\label{sec:Closing}
Let $R$ be a sufficiently large positive number such that all the critical values of the meromorphic function $f$ lie in the ball $\{|z|<R\}$.

In the holomorphic local parameter $\eta_j=y^{-1/k_j}$ in a vicinity $U_j$ ($|y_j|> R$) of the $j$-th conical end  of the angle $2\pi k_j$ ($k_j\geq 1$)
of the Riemannian manifold $(X, |df|^2)$ (i. e. a pole of $f$ of order $k_j$)
the metric $\bfm=|df|^2$ takes the form
$$\bfm=k_j^2\frac{|d\eta_j|^2}{|\eta_j|^{2k_j+2}}\ \ .$$

Let $\chi_j$ be a smooth function on ${\mathbb C}$ such that $\chi_j(\eta)=\chi_j(|\eta|)$, $|\chi_j(\eta)|\leq 1$,
$\chi_j(\eta)=0$ if $|\eta|>(R+1)^{-1/k_j}$, $\chi_j(\eta)=1$ if $|\eta|<(R+2)^{-1/k_j}$. Introduce the metric $\tbfm$
on $X$ such that
$$\tbfm=
\begin{cases} \bfm \ \ \ {\rm for}\ \ \ \ |z|<R\\
[1+(|\eta_j|^{2k_j+2}-1)\chi_j(\eta_j)]\bfm \ \ \ {\rm in}\ \ \ U_j\,.
\end{cases}
$$

Since the (Friedrichs extension of) the Laplace operator $\Delta^{\tbfm}$ has discrete spectrum and the corresponding operator $\zeta$-function is regular at $s=0$ (see e. g. \cite{K-LN}and references therein), one can define the determinant ${\det}^*\Delta^{\tbfm}$ via usual Ray-Singer zeta regularization.
Moreover, for  this determinant the usual BFK gluing formula (\cite{BFK}, Theorem $B^*$) holds (under the condition that the contour cutting the surface $X$ does not pass through
the conical singularities of the metric $\tbfm$).
Applying this standard BFK gluing formula, we get
\begin{equation}\label{smoothing}
\ln {\rm det}_\zeta^*\Delta^{\tbfm}=\ln C_0+\ln {\rm det}_\zeta \Delta_-^D+\ln {\rm det}^*_\zeta {\cal N}+\ln{\rm det}\Delta^{\tbfm}_{ext}\,,
\end{equation}
where $\Delta^{\tbfm}_{ext}$ is the operator of the Dirichlet problem for $\Delta^{\tbfm}$ in the union $\cup_j U_j.$ Using conformal invariance we see that
${\cal N}$ is the same as in Theorem 1 and $C_0$ is a moduli independent constant ($C_0=\frac{{\rm Area}(X, \tbfm)}{{\rm length}(\Sigma)}$).

Now equation (\ref{smoothing}) and Theorem 1 imply the following proposition.
\begin{proposition}\label{relord}
 The relative zeta regularized determinant $\det_\zeta (\Delta,\mathring\Delta)$ and the zeta-regularized determinant
${\rm det}_\zeta^*\Delta^{\tbfm}$ has the same variations with respect to moduli i. e. one has
\begin{equation}
\partial_{z_k}\ln \det_\zeta (\Delta,\mathring\Delta)=\partial_{z_k} \ln {\rm det}^*_\zeta\Delta^{\tbfm} \end{equation}
for $k=1, \dots, M$.
\end{proposition}

Thus, the relative determinant of the Laplacian on a noncompact surface  $(X, \bfm)$ with conical points and conical/Euclidean ends can be studied via consideration of the zeta-regularized determinant of Laplacian on a compact surface $(X, \tbfm)$ with conical points. The latter surface is flat everywhere except the conical singularities (whose positions vary when one changes the moduli $z_1, \dots, z_M$) and smooth ends of nonzero curvature which remain unchanged.

In the next two sections we study some spectral properties of compact surfaces with conical points.
The final goal is to derive the variational formulas for $\ln {\rm det}^*_\zeta\Delta^{\tbfm}$.

\section{$S$-matrix}\label{SM}
In this section we introduce the so-called $S$-matrix and relate its behavior at $\lambda=0$ with the Schiffer
projective connection. The definition of the $S$-matrix originates in
the general theory of boundary triplets (see \cite{Grubb} sect. 13)
and, more specifically, the general theory of self-adjoint extensions
of elliptic operators in singular settings (see \cite{NP}).  Here we
will follow closely \cite{HK}. However, we should point out that some normalization constants in the latter
reference are erroneous and that the normalization we use here is slightly different.

It is convenient to introduce the $S$-matric in the following general setting.
\subsection{General Setting and normalizations}
Let $(X,\tbfm)$ be a compact singular $2$d Riemannian manifold (possibly with boundary).
Let $P$ be an interior point of $X$ such that in a neighborhood $V$ of it $X$ is
isometric to a neighborhood of the tip of the Euclidean cone of angle $2\ell \pi.$

We set $X_0:= X\backslash \{ P \}$ and $X_\eps := X \backslash B(p,\eps).$
We also denote by $\gamma_r$ the circle of radius $r$ centered at $P$.

We will occasionnally use several different ways of parametrizing $V.$
\begin{itemize}
\item Polar coordinates $(r,\theta) \in (0,r_\ma)\times \R/ 2\ell \pi\Z,$
\item Local complex coordinate $z$. The $1$-form $dz$ is well-defined
on $V\backslash \{P \}$ and extends to $V$ as a holomorphic one form $\alp$ with a zero of
order $\ell-1$ at $P$. Note that it may not have a global holomorphic extension to $X$.
\item Distinguished complex parameter $y$ such that $\alp = \ell y^{\ell-1} dy$ near $P$.
\end{itemize}

We now want to consider the Laplace operator that is associated with $\tbfm.$
We assume that the set of singularities of $\tbfm$ consists of a finite number of conical points (in particular, it may consist of a single point
$P$).  Let $\Delta$  be the Friedrichs extension of the Laplace operator on $X$ with domain consisting of smooth functions that vanish near the singularities.

\begin{remark}
Actually the choice of extension (that is the prescription of a certain asymptotics at singular points to functions from the domain of the self-adjoint extension) should be made at each singularity (see \cite {HK}). Here we really care only about the choice of
extension at $P$, where we are choosing the Friedrichs extension (all the functions from the domain of this extension are bounded near $P$).  At other singularities one can chose any other extension, not necessarily the Friedrichs one.
\end{remark}

By definition we set $H^2(X)$ to be the domain of $\Delta$ and by $H^1(X)$ to be its form domain. We denote by $\Delta_0$
the restriction of $\Delta$ to functions in $H^2(X)$ that vanish near $p$ and
by $\Delta_0^*$ its formal adjoint. By choice,  the self-adjoint extension $\Delta$ corresponds to
the Friedrichs extension of $\Delta_0$. We will also denote by $H^2_0(X):= \dom \,\overline{\Delta_0}$.
Near $P$, we have
\[
\Delta_0^*=\, -4\partial_z \partial_{\overline{z}}
\,=\, -4\left( \ell^2|y|^{2(\ell-1)}\right)^{-1} \partial_y \partial_{\overline{y}}.
\]

Introduce a cut-off function $\rho$ such that $\rho$ has support in  $r\leq r_\ma$ and
equals $1$ near $r=0$. Define the functions $F^0$, $F^a_\nu$ and $F^h_\nu$ via
\begin{gather*}
F^0(z) \,= \,c_0 \ln (z\ovz)\rho(z)\,=\,c_0\ln (r^2)\rho(z),  \\
F^{a}_\nu(z) \,= c_\nu \overline{z}^{-\nu}\rho(z), ~~\nu= \frac{k}{\ell},~0<k<\ell\\
F^{h}_\nu(z) \,= c_\nu z^{-\nu}\rho(z), ~~\nu= \frac{k}{\ell},~0<k<\ell\,,\\
\end{gather*}
where
\begin{equation}\label{fixingc}
c_0 = \frac{1}{2\sqrt{\ell\pi}}, \quad
c_\nu =\frac{1}{2\sqrt{\nu \ell \pi}}.
\end{equation}

\begin{remark} The indices $a$ and $h$ correspond to ``antiholomorphic'' and ``holomorphic'' behavior of the corresponding functions $F$ at $0$.
In all the formulas below the index $\nu$ runs through the set $\{k/\ell\}_{k=1, \dots, \ell-1}$, where $2\pi \ell$ is the conical angle at $P$.
\end{remark}

Separating variables near $P$, it can be shown that any function in $\dom(\Delta_0^*)$ admits
the following expression (cf. \cite{K-LN}, \cite{Mo})
\begin{equation}\label{qe1}
\begin{split}
u = & \, c_0\Lambda^0(u) + c_0\Lambda^{0,-}(u) F^0(z) \\
& \,+\, \sum_{\nu} c_\nu \Lambda_\nu^{h,-}(u)z^{-\nu} \,+\, c_\nu \Lambda_\nu^{a,-}(u)\ovz^{-\nu}\\
&\,+\,\sum_\nu  c_\nu \Lambda_\nu^{h}(u)z^{\nu} \,+\, c_\nu \Lambda_\nu^{a}(u)\ovz^{-\nu} + u_0,
\end{split}
\end{equation}
where the $\Lambda$ are linear functionals on $\dom(\Delta_0^*)$ that vanish on $H^2_0$ and $u_0$
is in $H^2_0$.
Moreover, one has
\[
\ov{\Lambda_\nu^a(u)} \,=\, \Lambda_\nu^h(\ov{u}).
\]

For $u,v \in \dom(\Delta^*)$ we define
$$\Green(u,v):=\langle \Delta^* u, v\rangle -\langle u, \Delta^*v\rangle.$$
The Green formula implies
\begin{equation}
\Green(u,v) = \lim_{\eps\rightarrow 0} \frac{2}{i} \int_{\gamma_{\eps}} \partial_z u\ov{v} \,dz + u\partial_{\overline{z}}\ov{v} \,d\ovz,
\end{equation}\label{qe2}
where the circle $\gamma_\epsilon$ is positively oriented.
Since  $c_0$ and $c_\nu$ satisfy
$\displaystyle \frac{2c_0^2}{i}\int_{\gamma_\eps} \frac{dz}{z}=1$ and
$\displaystyle \frac{2\nu c_\nu^2}{i}\int_{\gamma_\eps} \frac{dz}{z}=1,$
the asymptotics (\ref{qe1}) and (\ref{qe2}) imply that
\begin{equation*}
\begin{split}
\Green(u,v) & = \Lambda^{0,-}(u)\ov{\Lambda^0(v)}-\Lambda^{0}(u)\ov{\Lambda^{0,-}(v)} \\
& \,+\,\sum_\nu \Lambda_\nu^{h,-}(u)\ov{\Lambda_\nu^{a}(v)} -\Lambda_\nu^a(u)\ov{\Lambda_\nu^{h,-}(v)} \\
& \,+\,\sum_\nu \Lambda_\nu^{h}(u)\ov{\Lambda_\nu^{a,-}(v)} -\Lambda_\nu^{a,-}(u)\ov{\Lambda_\nu^{h}(v)}. \\
\end{split}
\end{equation*}

The domain of the Friedrichs extension of $\Delta$ is characterized by requiring that all the coefficients with the superscript ``$-$'' vanish;  it follows from
 (\ref{qe2}) that the linear functionals $\Lambda^{a,h,0}_\nu$ are continuous over $H^2(X)$ and supported at $P$.
It can be proved that any linear functional $\Lambda$ that is continuous on $H^2(X)$ and supported at $P$ can be written
as a linear combination of the linear functionals $\Lambda_0,\,\Lambda_\nu ^a$ and $\Lambda_\nu^h.$
Finally, we notice that one has the following representation for the space $H^2_0(X)$:
\begin{equation*}
H^2_0(X) = \{ u\in H^2(X): ~ \Lambda^\sharp_\nu(u)= 0,~\forall \nu,\sharp=a,h~\;; \Lambda^0(u)=0\}\,.
\end{equation*}

\subsection{Definition of the $S$-matrix}
We follow \cite{HK}, paying special attention to conjugations and normalizing constants.

In the following the symbols $\sharp$ and $\dia$ are to be substituted by $0$, $h$ or $a$.
When the superscript is $0$ the subscript $\nu$ is $0$, when it is $a$ or $h,$ $\nu = \frac{j}{\ell}$ where $j$ ranges from $1$ to $\ell-1.$

We define
\begin{equation}
\begin{split}
f_\nu^\sharp(\cdot ;\,\lambda) & =\, (\Delta_0^* -\lambda)F_\nu^\sharp, \quad g_\nu^\sharp(\cdot ;\,\lambda) = -\left(\Delta-\lambda\right)^{-1} f_\nu^\sharp(\cdot ;\, \lambda), \\
G_\nu^\sharp(\cdot ; \,\lambda) &= F_\nu^\sharp \,+\, g_\nu^\sharp(\cdot ; \lambda), \quad
S^{\sharp \dia}_{\mu \nu}  = \Lambda^\sharp_\mu\left( g_\nu^\dia(\cdot ; \,\lambda) \right).\\
\end{split}
\end{equation}
Observe that by definition the $g$-functions belong to $H^2$, which makes the latter definition consistent when seeing
$\Lambda^\sharp_\mu$ as a linear functional over $H^2$.  Since $\Lambda^\sharp_\mu$ also makes sense as a linear functional
over $\ker (\Delta^*-\lambda)$, we may also write
\[
S^{\sharp \dia}_{\mu \nu}  = \Lambda^\sharp_\mu\left( G_\nu^\dia(\cdot ; \,\lambda) \right).
\]

\begin{remark}
The functions $F,~f,~g$ depend on the initial choice of $\rho$ but the linear functionals $\Lambda$ and the functions $G$ are cut-off independent.
\end{remark}

The $S$-matrix is defined by blocks :
\begin{equation}
S := \left(
\begin{array}{ccc}
S^{00} & S^{0h}_{0 \nu} & S^{0a}_{0 \nu} \\
S^{h0}_{\mu 0} & S^{hh}_{\mu\nu} & S^{ha}_{\mu\nu} \\
S^{a0}_{\mu 0} & S^{ah}_{\mu\nu} & S^{aa}_{\mu\nu}
\end{array}
\right).
\end{equation}
\hfill \\

\begin{remark}
If there are several conical points on the surface, then there are several ways to define the $S$-matrix depending on how
many points we want to take into account. The $S$-matrix defined in \cite{HK}  takes into account  all the conical points whereas the one defined here deals only with the conical point $P$ (even if there are other conical points on the surface).
Thus the $S$-matrix that is constructed here is only a part of the one from \cite{HK}.
\end{remark}

Applying Green's formula, we get
\begin{equation}\label{qe3}
\begin{split}
\Lambda^{\sharp}_\nu(u) & = \, \Green(u, \ov{F^{\sharp}})  = \, \dis \int (\Delta-\lambda)u\cdot  F_\nu^{\sharp} - u\cdot  f_{\nu}^\sharp( \cdot ; \lambda) dS\\
 & = \,\dis \int (\Delta-\lambda)u\cdot F_\nu^{\sharp} + u\cdot  (\Delta-\lambda)g_\nu^\sharp( \cdot ;\lambda) dS = \, \dis \int (\Delta-\lambda) u \cdot G_\nu^\sharp \, dS
\end{split}
\end{equation}
for any test function $u \in H^2(X)$ (here $dS$ is the area element on
$X$). We have used here that $\ov{F^{0,a,h}} = F^{0,h,a}\,\in
\dom(\Delta_0^*)$, $u$ and $g$ are in $H^2(X),$ and that $\Delta$ is
real (i.e. commutes with complex conjugation) and self-adjoint.

Applying (\ref{qe3}) to the $g$-functions gives the following alternative expressions for the $S$-matrix entries (we omit the dependence
on $\lambda$):
\begin{equation}
S^{\sharp\dia}_{\mu\nu} = \dis \int G^\sharp_\mu \left( \Delta-\lambda\right ) g_\nu^\dia  \, dS = \dis - \int G^\sharp_\mu f_\nu^\dia  \, dS.
\end{equation}

\begin{remark}
 The $S$-matrix allows the description of the elements of $\ker(\Delta_0^* -\lambda)$ in the following way.
For any element $u\in \dom(\Delta_0^*),$ denote by $L^{\pm}$ the column of its coefficients
$\Lambda^{\sharp,\pm}_\nu$ that describe the singular  behaviour of $u$ near $P$.  We have
\[
(\Delta_0^*-\lambda)u=0
\Leftrightarrow L^+=S(\lambda)L^-.
\]
This gives a (pure) formal analogy with a typical scattering
situation. 
In our setting, any solution to the equation $(\Delta_0^*-\lambda)u=0$ plays the role of scattered field, $L^\pm$ is the ``incoming'' and ``outgoing''
parts and the $S$-matrix  is the ``{\em scattering}'' matrix.
\end{remark}

\section{Basic properties of the $S$-matrix}

\subsubsection{Analyticity and complex conjugation}
From the analyticity of the resolvent we see that the $S$-matrix depends analytically on $\lambda$. Besides, the expression
\[
S^{\sharp\dia}_{\mu\nu}(\lambda)= \dis - \int G^\sharp_\mu(\cdot ; \lambda) f_\nu^\dia(\cdot; \lambda)  \, dS
\]
and the fact that
\begin{equation*}
\ov{f_\nu^{a,h}(\cdot,\lambda)} = f_\nu^{h,a}(\cdot; \ovl),\quad \ov{G^{a,h}(\cdot ; \lambda)}  = G_\nu^{h,a}(\cdot ; \ovl),
\end{equation*}
lead to the following identities :

\begin{equation}\label{eq:conj}
\ov{S^{hh}(\lambda)} \,=\,S^{aa}(\ovl), \quad
\ov{S^{ah}(\lambda)} \,=\, S^{ha}(\ovl).
\end{equation}

\subsubsection{Behavior for $\lambda$ going to $-\infty$}

For $r>0$ consider the equation
\[
-u''(r) -\frac{1}{r}u'(r) +\frac{\nu^2}{r^2}u(r)\,=\,\lambda u(r).
\]
Since $\nu\neq 0$, any solution to this equation has the following asymptotic behaviour near
zero :
\[
u(r) \,=\, a_-r^{-\nu}\,+\,a_+r^\nu +o(r^\nu).
\]
and the vector space of solutions that belongs to $L^2(rdr)$ is one-dimensional.
We set $k_\nu(r;\lambda)$ to be the unique solution to this equation which is in $L^2(rdr)$ and
that is normalized in such a way that
\begin{gather*}
F_\nu^{h}(r,\theta)-k_\nu(r;\lambda)\exp(-i \nu \theta) = O(r^\nu),\\
F_\nu^{a}(r,\theta)-k_\nu(r;\lambda)\exp(i \nu \theta) = O(r^\nu)
\end{gather*}
(i.e. we adjust the coefficient of $r^{-\nu}$ in $k_\nu^\sharp$ so that it coincides with the coefficient of $F_\nu^\sharp$).
By inserting a cut-off $\rho$, we define
\begin{gather*}
K_\nu^{h}(r,\theta ; \lambda) := k_\nu(r;\lambda)\exp(-i \nu\theta)\rho(r),\\
K_\nu^{a}(r,\theta ; \lambda) := k_\nu(r;\lambda)\exp(i\nu\theta)\rho(r)\\
\end{gather*}
as  functions on $X$. We compute $R_\nu^\sharp( \ \cdot\  ; \lambda) \,=\,\left(
  \Delta_0^*-\lambda\right ) K_\nu^{\sharp}$, where $\sharp=a,h$.

\begin{lemma}
For $\sharp = a,h$ we have
\begin{equation*}
G_\nu^\sharp(\  \cdot\  ; \lambda)=\, K_\nu^{\sharp} - \left[ \Delta-\lambda\right]^{-1}R_\nu^\sharp .
\end{equation*}
\end{lemma}

\begin{proof}
By construction it is straightforward that both sides of the equation
are in  $\ker(\Delta_0^*-\lambda)$ and by choice of normalization, both share the same
singular behaviour.
\end{proof}

We define by $\kappa_\nu(\lambda)$ the coefficient of $r^\nu$ in the asymptotic expansion
of $k_\nu(\ \cdot\ ;\lambda)$.
\begin{corollary}
As $\Re \lambda $  goes to $-\infty$ we have
\begin{equation*}
\begin{split}
S^{hh}(\lambda)& =\,O(|\lambda|^{-\infty}),\quad
S^{aa}(\lambda)=\,O(|\lambda|^{-\infty}),\\
S^{ah}(\lambda) &=\, \mathrm{diag}(\kappa_\nu(\lambda))\,+\,O(|\lambda|^{-\infty}),\\
S^{ha}(\lambda)& =\,\mathrm{diag}(\kappa_\nu(\lambda))\,+\,O(|\lambda|^{-\infty}).\\
\end{split}
\end{equation*}
\end{corollary}

\begin{proof}
The asymptotic expansion of Bessel functions implies  that
$\| R_\nu^\flat \|_{L^2} = O(|\lambda|^{-\infty})$.  Therefore
$\Lambda_\nu^\sharp\left(\left[ \Delta-\lambda\right]^{-1}R_\nu^\flat\right)\,=\,O(|\lambda|^{-\infty}).$
Thus all entries of the $S$-matrix are given by
$\Lambda_\mu^\sharp(K_\nu^\flat)$ up to $O(|\lambda|^{-\infty}).$ The
first term is seen to be $0$ except for the diagonal terms in $S^{ah}$
or $S^{ha}$ for which it is $\kappa_\nu(\lambda).$
\end{proof}

\begin{remark}
A different proof is given in \cite{HK} using the heat kernel.
\end{remark}

\subsubsection{Differentiation with respect to $\lambda$}
We denote by a dot the differentiation with respect to $\lambda.$
Differentiating the defining equation for $g^\sharp_\nu$, we find that
\[
\dot{g}^\dia_\nu = \left( \Delta-\lambda\right)^{-1} G^{\dia}_\nu.
\]
Thus, we get
\begin{equation}\label{eq:dotS}
\dot{S}^{\sharp \dia}_{\mu \nu}  = \Lambda_\mu^\sharp\left( \dot{g}^\dia_\nu \right)= \dis \int G^\sharp_\mu G^\dia_\nu \, dS.
\end{equation}
From this relation we deduce the following proposition.

\begin{prop}
For any $\lambda$, $S^{aa}(\lambda)$ and $S^{hh}(\lambda)$ are symmetric matrices and
\begin{equation}\label{eq:transpose}
{}^tS^{ah}(\lambda)= S^{ha}(\lambda).
\end{equation}
\end{prop}

\begin{proof}
The expression for $\dot{S}^{\sharp \dia}_{\mu \nu}$ yields that $\dot{S}^{aa}$ and $\dot{S}^{hh}$ are symmetric matrices.
Since both tend to symmetric matrices (actually $0$) as $\lambda$ goes to $-\infty$, the first part of the claim follows.
In the same way we obtain
\begin{equation*}
\begin{split}
\dot{S}^{ah}_{\mu\nu} -\dot{S}^{ha}_{\nu\mu} = 0.
\end{split}
\end{equation*}
Since $S^{ah}$ and ${}^t S^{ha}$ tend to the same diagonal matrix as $\lambda$ goes to $-\infty$, the
second part of the claim also follows.
\end{proof}

Combining the identities (\ref{eq:conj}) and $(\ref{eq:transpose})$ we conclude that $S^{ah}$ is hermitian for $\lambda$ real; actually it is an analytic family of hermitian matrices, meaning that $S^{ah}(\lambda)=\left(S^{ah}(\ovl)\right)^*$.

\subsubsection{Behavior for $\lambda$ going to $0$}
The matrix $S(\lambda)$ is well defined a priori only for $\lambda$ in the resolvent set of $\Delta.$
However it is always possible to define the function $f_\nu^\sharp(\ \cdot\ ;\,\lambda=0).$ Whenever $\sharp\neq 0$
the latter function is in the range of $\Delta.$ We can thus find solutions $g_\nu^\sharp(\cdot;0)$ to the equation
\[
\Delta g_\nu^*= -f_\nu^*(\cdot, 0).
\]
The latter solutions are defined only up to the addition of a constant. It follows that the definition of $S^{\sharp \dia}_{\mu\nu}(0)$ makes sense for $\sharp\neq 0$, $\dia \neq 0$
 and  the following Proposition holds.

\begin{proposition}
For $\sharp\neq 0$ and $\dia \neq 0$ the matrix-valued function $\lambda \mapsto S^{\sharp \dia}(\lambda)$ extends holomorphically to a neighbourhood of $0.$  Moreover $S^{\sharp \dia}(0)$ depends only on the conformal class of $\tbfm.$
\end{proposition}

\begin{proof}
For $\lambda$ close to $0$ we have
\[
g_\nu^\sharp \,=\, \frac{1}{\lambda}\int_X f_{\nu}(\cdot \,,\,\lambda)\,dS\,+\, g_{\nu}^{\sharp,\perp}(\cdot\,,\,\lambda),
\]
where $\int_X g_{\nu}^{\sharp,\perp} =0.$
Since $\lambda \mapsto \int_X f_\nu^\sharp(\lambda ; 0)  =0$ is
holomorphic and vanish at $0$, we obtain that $\lambda \mapsto G_\nu^*$ can be holomorphically continued to a neighbourhood of $0.$
The first statement follows. The second statement follows by remarking that $G_\nu^\sharp$ is
a function in $\dom(\Delta_0^*)$ such that $\Delta_0^* G_\nu^\sharp =0$ and the singular behaviour near $P$ is
prescribed. Both conditions are conformally invariant so that if we change the metric in its conformal class, we may only change $G$ by adding a constant. This will not affect the coefficients in the $S$-matrix we are considering here.
\end{proof}

\subsection{$S(0)$ and the Schiffer projective connection}

Chose a marking for the Riemann surface $X$, i. e. the canonical basis $a_1, b_1, \dots, a_g, b_g$ of $H_1(X, {\mathbf Z})$. Let $\{v_1, \dots,
v_g\}$ be the basis of holomorphic differentials on $X$ normalized via
$$\int_{a_i}v_j=\delta_{ij}\,.$$
Then the matrix of $b$-periods of the marked Riemann surface $X$ is defined via
$${\mathbf B}=||\int_{b_i}v_j||\,.$$

Let $W(\,\cdot\,,\,\cdot\,)$ be the canonical meromorphic bidifferential on $X\times X$, with properties $W(P,Q)=W(Q, P)$,
$$\int_{a_i}W(\,\cdot\,, P)=0,$$
and
$$\int_{b_j}W(\,\cdot\,, P)=2\pi iv_j(P).$$
The bidifferential $W$ has the only double pole along the diagonal $P=Q$. In any holomorphic local parameter $x(P)$ one has the asymptotics
\begin{equation}\label{funH}W(x(P), x(Q))=\left(\frac{1}{(x(P)-x(Q))^2}+H(x(P), x(Q))\right)dx(P)dx(Q),\end{equation}
$$H(x(P), x(Q))=\frac{1}{6}S(x(P))+O(x(P)-x(Q)),$$
as $Q\to P$, where $S_B(\cdot)$ is the Bergman projective connection.


Consider the Schiffer bidifferential
$${\cal S}(P, Q)=W(P, Q)-\pi\sum_{i, j}(\Im {\mathbf B})^{-1}_{ij}v_i(P)v_j(Q).$$
The Schiffer projective connection, $S_{Sch}$, is defined via the asymptotic expansion
$${\cal S}(x(P), x(Q))=\left(\frac{1}{(x(P)-x(Q))^2}+\frac{1}{6}S_{Sch}(x(P))+O(x(P)-x(Q))\right)dx(P)dx(Q)\,$$
One has the equality
\begin{equation}\label{connection}
S_{Sch}(x)=S_B(x)-6\pi \sum_{i, j}(\Im {\mathbf B})^{-1}_{ij}v_i(x)v_j(x)\,.\end{equation}

In contrast to the canonical meromorphic differential and the Bergman projective connection, the Schiffer bidifferential and the Schiffer projective connection are independent of the marking of the Riemann surface $X$.

Introduce also the so-called Bergman kernel (which is in fact the Bergman reproducing kernel for holomorphic differentials on $X$) as
$$B(x, \bar x)=\sum_{ij}(\Im {\mathbf B})^{-1}_{ij}v_i(x)\overline{v_j(x)}\,.$$

\begin{prop}\label{clue}
Let $X$ be a Riemann surface and let $\tbfm$ be a conformal metric on
$X$, suppose that $\tbfm$ has a conical singularity of angle $2\ell\pi$ at $p.$
Let also $x$ be the distinguished local parameter for $\tbfm$ near
$p$. Then there is the following relation between the entries of the
holomorphic-holomorphic part, $S^{hh}(0)$, of the $S$-matrix :
\begin{equation}\label{eq:clue}
\sum_{k=1}^{\ell-1}\frac{\sqrt{k(\ell-k)}}{\ell}S^{hh}_{\frac{\ell}{\ell}\frac{\ell-k}{\ell}}(0)=-\frac{1}{6\ell(\ell-2)!}\left(\frac{d}{dx}\right)^{\ell-2}S_{Sch}(x)\Big|_{x=0}\,.
\end{equation}
\end{prop}

\begin{remark}
The same would hold true for a conical singularity of angle $\beta$ with $2\pi (\ell-1)<\beta\leq 2\pi \ell$.
\end{remark}

\begin{remark}\label{rk:A}
Observe that using the indices $\mu,\nu =\frac{k}{\ell}$ the left-hand-side of  \eqref{eq:clue} can be written  as
\[
\sum_{\mu+\nu=1} \sqrt{\mu}\sqrt{\nu}S^{hh}_{\mu\nu}(0).
\]
\end{remark}

\begin{proof} Introduce the following one forms $\Omega_k$ and $\Sigma_k$ on $X$:
$$\Omega_k=-\frac{1}{(k-1)!}\left(\frac{d}{dx}\right)^{k-1}\frac{W(\,\cdot\,, x)}{dx}\Big|_{x=0}+\frac{2\pi i}{(k-1)!}\sum_{\alpha, \beta}
\left(\Im {\mathbf B} \right)^{-1}_{\alpha \beta}\left\{\Im v^{(k-1)}_\beta(0)\right\}v_\alpha(\cdot)$$

$$\Sigma_k=-i\frac{1}{(k-1)!}\left(\frac{d}{dx}\right)^{k-1}\frac{W(\,\cdot\,, x)}{dx}\Big|_{x=0}+\frac{2\pi i}{(k-1)!}\sum_{\alpha, \beta}
\left(\Im {\mathbf B} \right)^{-1}_{\alpha \beta}\left\{\Re v^{(k-1)}_\beta(0)\right\}v_\alpha(\cdot)\,,$$
where
$$v^{(k-1)}_\beta(0):=\left(\frac{d}{dx}\right)^{k-1}\frac{v_\beta(x)}{dx}\Big|_{x=0}\,.$$

All the periods of the differentials $\Omega_k$ and $\Sigma_k$ are pure imaginary, therefore,
one can correctly define the function $f_k$ on $X$ via
$$f_k(Q)=\Re\left\{ \int_{P_0}^Q \Omega_k  \right\}-i\Re\left\{\int_{P_0}^Q\Sigma_k\right\}\,$$
where $P_0$ is an arbitrary base point not coinciding with $P$.
Clearly, $f_k$ is harmonic in $X\setminus\{P\}$ and
\begin{equation}\label{razlozh}f_k(x)=\frac{1}{x^k}+{\rm const}+\sum_{j=1}^{\infty}(c_jx^j+d_j\bar x^j)\end{equation}
in a vicinity of $P$. One gets
$$c_l=-\frac{1}{l!(k-1)!}\partial_x^{l-1}\partial_y^{k-1}H(x, y)\Big|_{x=y=0}
 +\frac{\pi}{l!(k-1)!}\sum_{\alpha, \beta}\left(\Im {\mathbf B} \right)^{-1}_{\alpha \beta}v^{(k-1)}_\beta(0)v^{(l-1)}_\alpha(0)$$
and
$$S^{hh}_{\frac{k}{\ell}\frac{l}{\ell}}(0)=\sqrt{\frac{l}{k}}\ c_l.$$

This implies that
\begin{equation*}
\begin{split}
\sum_{k=1}^{\ell-1}\frac{\sqrt{k(\ell-k)}}{\ell}S^{hh}_{\frac{k}{\ell}\frac{\ell-k}{\ell}}(0)&\,=
-\frac{1}{\ell}\sum_{k=0}^{\ell-2}\frac{1}{k!(\ell-2-k)!}\partial_x^{\ell-2-k}\partial_y^{k}H(x,
y)\Big|_{x=y=0} \\
&
\,+\frac{\pi}{\ell(\ell-2)!}\left(\frac{d}{dx}\right)^{\ell-2}\sum_{\alpha,
  \beta}\left(\Im{\mathbf B}\right)^{-1}_{\alpha \beta}
\frac{v_\alpha(x)v_\beta(x)}{(dx)^2}\Big|_{x=0}\,.
\end{split}
\end{equation*}

Since
$$\frac{1}{6}S_B(x)=H(x, x)=\sum_{n=0}^\infty\frac{1}{n!}(\partial_x+\partial_y)^nH(x, y)\Big|_{x=y}x^n,$$
we have
$$\frac{1}{6}S_B^{(n)}(0)=\sum_{p=0}^n\frac{n!}{p!(n-p)!}\partial_x^p\partial_y^{n-p}H(0, 0),$$
which implies the proposition.
\end{proof}

\begin{remark}From (\ref{razlozh}) with $k=1$ it follows that for conical angles $2\pi<\beta\leq 4\pi$ we have
$$\left(\begin{matrix}
S^{hh}(0)\ \ S^{ha}(0)\\
S^{ah}(0)\ \ S^{aa}(0)
\end{matrix} \right)= \left(\begin{matrix}
-\frac{1}{6}S_{Sch}(0)\ \ B(0, 0)\\
B(0, 0)\ \ -\frac{1}{6}\overline {S_{Sch}(0)}\,
\end{matrix} \right)\,,
$$
where the Schiffer projective connection and the Bergman kernel are calculated in the distinguished local parameter at $P$.

\end{remark}

\section{Variational formulas with respect to moduli}
In this section we derive the variational formulas for $\ln {\rm det}
\Delta^{\tbfm}$. This derivation goes as follows. First,
using Kato-Rellich theory (see \cite{Kato}),
we prove variational formulas for the individual eigenvalues of
the operator $\Delta^{\tbfm}$. Using these formulas and
the contour integral representation of the zeta-function of
$\Delta^{\tbfm}$, we express the variations of the value
$\zeta_{\Delta^{\tbfm}}'(0)$
with respect to the critical value $z_k$ through a combination of the
matrix elements of the $S$-matrix at the conical point $P_k$
(the zero of the meromorphic differential $df$) of the metric
$\tbfm$. The latter combination is the one appearing in Proposition
\ref{clue} and can be expressed through the Schiffer projective connection.

\subsection{Variational formula for eigenvalues of $\Delta^{\tbfm}$}
\begin{remark}
In this section we will use $w$ for the moduli parameter and on the surface we will use
the complex parameter $z$ and $(x,y)$ for the associated local cartesian coordinates
(so that $z=x+iy$). We warn the reader that in the rest of the paper we use $z_i$ as the moduli
parameters and $x$ as a local complex parameter on $X.$
\end{remark}

\subsubsection{Moving conical points}\label{sss}
Let $\tbfm$ be a metric as constructed in section \ref{sec:Closing}. Let $P$ be one of its conical points.
We wish to define a metric $\tbfm_w$ corresponding to the shift of $P$ by  $w\in \C$.
The following makes this construction precise.

Let $\C$ be the complex plane with pointed origin. We set $\tilde{X}_{w}$ to be the $\ell$-fold covering of $\C$
with one ramification point at $w$ so that $\tilde{X}_w$ can be identified with the Euclidean cone
of total angle $2 \ell \pi.$

Fix a cutoff function $\rho$ and define a map $\phi_w$ from $\C$ to itself by
\[
\phi_w(z) = z\,+\,\rho(|z|)w.
\]
For $w$ small enough, this defines a family of smooth selfdiffeomorphisms of $\C$.
The cone $\tilde{X}_0$ can be obtained by gluing together $\ell$ copies of the plane after cutting along a fixed half-line $d$ that emanates
from the origin.  The cone $\tilde{X}_w$ can then be obtained by gluing $\ell$ copies of $\C$ after cutting it along $\phi_w(d).$

The function $\phi_w$ thus defines a family of smooth diffeomorphisms from $\tilde{X}_0$ onto $\tilde{X}_w$.
Let  the metric $g_w$ on $\tilde{X}_0$ be the pull-back of the Euclidean metric on $\tilde{X}_w$ by $\phi_w$.

We write $w=a+ib$ and use the local cartesian coordinates $x+iy=z$ near $P$.
 For the metric $g_w = A(x,y; w) dx^2 + 2B(x,y ; w) dxdy + C(x,y ; w)dy^2$ we obtain the following
expressions :
\begin{equation}\label{ploho}
\left( \begin{array}{cc}
A & B \\
B & C
\end{array}
\right)  = {}^tD\phi_w D\phi_w, \quad D\phi_w  = \left( \begin{array}{cc}
1 + \frac{ax}{r}\rho'_1(r) & \frac{ay}{r}\rho_1'(r) \\
\frac{bx}{r}\rho_1'(r) & 1+\frac{by}{r}\rho_1'(r)
\end{array}
\right).
\end{equation}
It follows by direct verification  that the coefficients of $g_w$ are polynomials in $a$, $b$.
Observe that $g_w$ coincide with $g_0$ outside a ball centered at $P$ so that $g_w$ can be smoothly extended by
any Riemannian metric that coincides with the Euclidean one in an annulus centered at $p.$
This allows us to define a metric $\tbfm_w$ on our given setting $X\equiv X_0$ that corresponds to some $X_w$ that
is obtained by fixing the exterior of a small ball centered at $P\in X$ and, in an even smaller ball, by shifting the conical point by
$w$.

We denote by $J_w$ the jacobian determinant of the metric $\tbfm_w$
 on $X$, by $q_w$ the Dirichlet energy quadratic form associated
with $\tbfm_w$, and by $n_w$ the Riemannian $L^2(X)$ scalar product on $(X, \tbfm_w)$.

We thus have the following expressions (for a real $u$ that is supported near $P$)
\begin{equation}
\begin{split}
q_w(u) & = \int_{X_0} \left[ C (\partial_x u)^2 -2B\partial_x u \partial_y v + A (\partial_y u)^2\right ] J_w^{-\und} dxdy,  \\
n_w(u) & = \int u^2  J_w^{\und} dxdy.
\end{split}
\end{equation}
Observe that  $q_w(u)$ and $n_w(u)$ do not depend on $w$ for $u$ supported away of $P$.

In order to apply spectral perturbation theory, we will need the first order variations of $q_w(u)$ and $n_w(u)$.
We prove the relevant lemma below.

\begin{lemma}\label{lem:varq}
For any $\lambda\in\Bbb C$ and any $u\in H^1(X)$ we have
\begin{equation}
\begin{split}
\left[ -\partial_w q+\lambda \partial_w n\right]_{w=0}(u) &= \, 2\int_{X_0} (\partial_z u)^2 \, \frac{z\rho'(r)}{r}  \, dxdy
\,+ \frac{\lambda}{2} \int_{X_0} u^2 \, \frac{\ov{z}\rho'(r)}{r}  \, dxdy,\\
\left[ -\partial_{\ov{w}} q+\lambda \partial_{\ov{w}} n\right]_{w=0}(u) &= \, 2\int_{X_0} (\partial_{\ov{z}} u)^2 \, \frac{\ov{z}\rho'(r)}{r}  \, dxdy
\,+ \frac{\lambda}{2} \int_{X_0} u^2 \, \frac{z\rho'(r)}{r}  \, dxdy.\\
\end{split}
\end{equation}
\end{lemma}

\begin{proof}
Denote by
\[
\mathbb{G}_w := \left( \begin{array}{cc}
C & -B \\
-B & A
\end{array}
\right)
\]
so that we have
\[
q_w(u) := \int_{X_0} {}^t\nabla u \Gbb \nabla u \cdot J_w^{-\und} dxdy.
\]
Differentiating at $w=0$, we obtain
\begin{equation*}
\pw q_w(u) = \int_{X_0} {}^t \nabla u \cdot \left( \pw \Gbb -\frac{1}{2} \pw J \Ibb\right )\cdot \nabla u  dS.
\end{equation*}
A straightforward computation yields
\begin{equation*}
\begin{split}
\partial_a \Gbb -\frac{1}{2} \partial_a J \Ibb & =
\left( \begin{array}{cc}
-\frac{x\rho'}{r} & -\frac{y\rho'}{r} \\
-\frac{y\rho'}{r} & \frac{x\rho'}{r}
\end{array} \right) \\
\partial_b \Gbb -\frac{1}{2} \partial_b J \Ibb & =
\left( \begin{array}{cc}
\frac{y\rho'}{r} & -\frac{x\rho'}{r} \\
-\frac{x\rho'}{r} & -\frac{y\rho'}{r}
\end{array} \right) \\
\end{split}
\end{equation*}
From this we find
\begin{equation*}
\begin{split}
{}^t \nabla u \cdot \left( \pw \Gbb -\frac{1}{2} \pw J \Ibb\right )\cdot \nabla u & = {}^t \nabla u\cdot %
\left( \begin{array}{cc}
-\frac{z\rho'}{2r} & \frac{iz\rho'}{2r} \\
\frac{iz\rho'}{2r} & \frac{z\rho'}{2r}
\end{array} \right)
\cdot \nabla u \\
&= -\frac{z\rho'}{2r} \left(  (\partial_x u)^2-(\partial_y u)^2 -2i \partial_xu \partial_y u \right) \\
&= -\frac{2z \rho'}{r} (\partial_z u)^2.
\end{split}
\end{equation*}
The other terms proceed in the same way.
\end{proof}

\subsubsection {Variational formulas for eigenvalues of $\tbfm$}
In this section, we compute variational formulas for the eigenvalues of $\tbfm_w.$
In order to do so we use the Kato-Rellich analytic perturbation
theory~\cite{Kato}. It should be noticed that the family of metrics $g_w$ is smooth in $w$ but {\it not analytic}, see \eqref{ploho}. We thus  fix $w$ and introduce
$q_t = q_{tw}$ and $n_t = n_{tw}$, which are analytic in $t$. In this way we may only consider directional derivatives.

The eigenvalue equation that gives the spectrum of $q_t$ relatively to $n_t$ is
\begin{equation}\label{eq:eigeqK}
q_t(u_t,v) \,= \,\lambda_t n_t(u_t,v).
\end{equation}
This problem is analytic in $t$ so that the eigenvalues are organized
into real-analytic branches, see \cite[Chapter VII.6.5]{Kato} for details.

The first-order variation for the eigenbranch $(\lambda_t,u_t)$ is given by the following Feynman-Hellmann formula
\begin{equation}\label{eq:FeynHel}
\frac{d\lambda}{dt} = \frac{dq}{dt}(u) -\lambda\frac{dn}{dt}(u)
\end{equation}
which is obtained by differentiating eq. \eqref{eq:eigeqK} with $v$
fixed and then evaluating at $v=u_t.$

\begin{proposition}
Let $r$ be small enough, then for $\lambda\in \spec(\Delta^{\tbfm})$ we have
\begin{equation}\label{eq:vareig}
\begin{split}
\partial_w \lambda & = \frac{2}{i} \int_{\gamma_r} (\partial_z u)^2 dz - \frac{\lambda}{4} u^2 d \ovz \\
\povw \lambda & = -\frac{2}{i} \int_{\gamma_r} (\povz u)^2 d\ovz -\frac{\lambda}{4} u^2 dz.
\end{split}
\end{equation}

Let $(\lambda_t,u_t)$ be an eigenbranch of $q_t$  relatively to $n_t$, then
$\lambda'= \frac{d}{dt}_{|{t=0}} \lambda$ is given by
\[
\lambda' = w \pw \lambda + \overline{w} \povw \lambda,
\]
where in the expression of $\pw \lambda$ and $\povw \lambda,$ $u=u_0$ is the eigenvector of the eigenbranch $(\lambda_t,u_t)$ at $t=0$ .
\end{proposition}

\begin{remark}
We remind the reader of one subtlety of perturbation theory (see
\cite{Kato}, \cite{KM}). In case of a multiple eigenvalue $\lambda_0$, for any family $q_t$ there are several eigenbranches emanating
from $\lambda_0,$ and the initial corresponding eigenvectors may actually depend of the
chosen family. In particular the expressions $\partial_w \lambda$ and $\partial_{\ov{w}}\lambda$
also depend on the initial $w$ that defines $q_t:=q_{tw}$. In other terms, for any direction $w$ it is possible
to organize the spectrum into eigenvalues branches but it may not be possible
to organize the eigenvalues as functions that are differentiable with respect to $w$ varying in the ball.
\end{remark}

\begin{proof}
We start with the one form
\[
\omega_u\,=\rho(z)\cdot \left((\partial_z u)^2dz-\frac{\lambda}{4} u^2 d\overline{z}\right).
\]
Since $(\lambda,u)$ is an eigenpair of the Laplace operator we compute
\[
d\omega_u = -\left [ \frac{z\rho'}{2r}(\partial_zu)^2 + \frac{\lambda}{4}\cdot \frac{\ovz\rho'}{2r}u^2\right] dz\wedge d\ovz.
\]
We now use Stokes formula to obtain
\begin{equation*}
\begin{split}
\int_{\gamma_r} \omega_u &= -\int_X d\omega_u \\
& = \int_X \left [ \frac{z\rho'}{2r}(\partial_zu)^2 + \frac{\lambda}{4}\cdot \frac{\ovz\rho'}{2r}u^2\right] dz\wedge d\ovz\\
& = \frac{1}{2} \int_X \left[ \frac{z\rho'}{r}(\partial_zu)^2 + \frac{\lambda}{4}\cdot \frac{\ovz\rho'}{r}u^2\right] (-2i dxdy).
\end{split}
\end{equation*}
On the other hand, in (\ref{eq:FeynHel}) we use the formulas provided by Lemma \ref{lem:varq} to obtain~:
\begin{equation}\label{eq:diffEstart}
\begin{split}
\partial_w \lambda&= -\, 2\int_{X_0} (\partial_z u)^2 \, \frac{z\rho'(r)}{r}  \, dxdy
\,- \frac{\lambda}{2} \int_{X_0} u^2 \, \frac{\ov{z}\rho'(r)}{r}  \, dxdy.\\
\end{split}
\end{equation}
Comparing the two yields the first formula in~\eqref{eq:vareig}.
The second one follows either from the same computation or by complex conjugation.
\end{proof}

Since $\omega_u$ is closed in $B(p,r_0)\setminus \{p\}$ we may let tend $r$ to $0$ in the preceding formulas.
We thus obtained a formula for $\partial_w \lambda$ that is expressed only through the asymptotic expansion of $u$ near
$p.$

Recall that by definition of the linear functionals $\Lambda_\nu^\sharp,$ we have in the local coordinate $z$ the following expansion
near $p$
\[
u(z) := c_0 \Lambda^0(u) + \sum_\nu c_\nu \Lambda^h_\nu(u)z^\nu + c_\nu \Lambda^a(u)\ovz^\nu\,+\,u_0.
\]
with $u_0\in H^2_0.$ By letting $r$ go to zero we obtain the following lemma.

\begin{lemma}\label{lem:defA}
Let $A=[a_{\mu\nu}]$ be the matrix defined by
\begin{equation*}
\left \{\begin{array}{ll}
a_{\mu\nu} = 4\pi \mu c_\mu \cdot \nu c_\nu & \mbox{if}~ \mu+\nu=1,\\
a_{\mu\nu} = 0 & \mbox{otherwise}.
\end{array}
\right.
\end{equation*}
We have the alternative expressions
\begin{equation}
\pw \lambda = \sum_{\mu + \nu =1} \Lambda_\nu^h(u) a_{\mu \nu} \Lambda_\mu^h(u),\quad
\povw \lambda =   \sum_{\mu + \nu =1} \Lambda_\nu^a(u) a_{\mu \nu} \Lambda_\mu^a(u).
\end{equation}
\end{lemma}

\begin{proof}
We prove the formula for $\partial_w \lambda,$ the proof is the same for $\partial_{\ovw}\lambda.$
First observe that since $u$ is bounded we have
\[
\lim_{r\rightarrow 0} \int_{\gamma_r} u^2 d\ovz =0.
\]
Now, if $u_0$ is smooth and compactly supported away of $p,$ using
Stokes' formula, we have
that for any $v\in H^2$
\[
\int_{\gamma_r} \partial_z v \partial_z u_0 dz = \frac{1}{4} \int_{B_r} \Delta v \partial_z u_0 \,+\,\partial_z v \Delta u_0 \, dz\wedge d\ovz.
\]
By continuity, this equality persists for $u_0\in H^2_0.$ It follows that for any $u\in H^2$ and any $u_0\in H^2_0$ we have
\[
\lim_{r\rightarrow 0} \int_{\gamma_r} \partial_z u \partial_z u_0 dz =0.
\]
It follows that
\[
\partial_w \lambda = \lim_{r\rightarrow 0} \int_{\gamma_r} \left( \partial_z(u-u_0)\right)^2\,dz.
\]
By definition we have
\[
u-u_0 \,=\,c_0 \Lambda^0(u) + \sum_\nu c_\nu \Lambda^h_\nu(u)z^\nu + c_\nu \Lambda^a(u)\ovz^\nu,
\]
so that the claim follows by a direct computation.
\end{proof}

Using the definition of $\Lambda_\nu^h$ and the fact that $u$ is an eigenfunction, we obtain

\begin{corollary}\label{coro:acv}
For any  $\lambda \in \C \setminus [0,\infty)$ the series
$\sum_{\lambda_n\in \spec(\Delta^{\tbfm})} \partial_w \lambda_n (\lambda_n-\lambda)^{-2}$ is absolutely convergent and
\[
\sum_{\lambda_n\in \spec(\Delta^{\tbfm})}  \frac{\partial \lambda_n}{(\lambda_n-\lambda)^2} \,=\, \Tr \left ( A\frac{\partial S^{hh}}{\partial \lambda} (\lambda)\right).
\]
 \end{corollary}

 \begin{proof}
  To prove the absolute convergence it suffices to show that for any $\nu$ we have
\[
\sum_{\lambda_n\in \spec(\Delta^{\tbfm})}  \left | \frac{\Lambda^h_\nu(u_n)}{\lambda_n-\lambda} \right |^2 \,<\,\infty.
\]
Since
\[
\Lambda^h_\nu(u_n) \,= \, \int (\Delta-\lambda)u_n G^h_\nu \,=\, (\lambda_n-\lambda) \langle G^h_\nu, u\rangle,
\]
the claim follows by remarking that the eigenfunctions $u_n$ form an orthonormal basis.
By Plancherel formula, we then obtain
\[
\sum_{\lambda_n\in \spec(\Delta^{\tbfm})}  \frac{\partial_w \lambda_n}{(\lambda_n-\lambda)^2} \,=\,\sum_{\mu,\nu} a_{\mu\nu} \int_X G^h_\mu(x\,;\,\lambda) G^h_\nu(x\,;\,\lambda)\,dS.
\]
We now remark that using \eqref{eq:dotS}
\[
 \int_X G^h_\mu(x\,;\,\lambda) G^h_\nu(x\,;\,\lambda)\,dS \,=\,\frac{\partial_\lambda S^{hh}_{\mu\nu}}{\partial \lambda},
\]
and $A$ is a symmetric matrix.
 \end{proof}

 \begin{remark} For
$\sum_{\lambda_n\in \spec(\Delta^{\tbfm})}  {\partial \lambda_n}{(\lambda_n-\lambda)^{-2}}$
we also have a similar formula involving $S^{aa}$.
 \end{remark}

\subsection{Variational formula for $\zeta'(0\,;\,\Delta^{\tbfm} )$}
We prove the following proposition.

\begin{proposition}
Let $\tbfm_w$ be the family of metrics defined above, $S^{hh}$ be the holomorphic-holomorphic part
of the corresponding $S$-matrix and $A$ be the matrix defined in Lemma \ref{lem:defA}. We have
\begin{equation}
\partial_w \zeta'(0\,;\,\Delta^{\tbfm})\,=\, \Tr\left (AS^{hh}(0)\right)\,=\,\sum_{\mu+\nu=1} \sqrt{\mu}\sqrt{\nu}S^{hh}_{\mu\nu}(0).
\end{equation}
\end{proposition}

\begin{proof}
We start from the following integral representation of the zeta-function of the operator $\Delta^{\tbfm}-\lambda$ through the trace of the second
 power of the resolvent:
\begin{equation}\label{eq:defzetcont}
s\zeta(s+1; \Delta^{\tbfm}-\lambda)=\frac{1}{2\pi i}\int_{\Gamma_\lambda}(z-\lambda)^{-s}{\rm Tr}\left((\Delta^{\tbfm}-z)^{-2}\right)dz,
\end{equation}
where $\Gamma_\lambda$ is a contour connecting $-\infty+i\epsilon$ with $-\infty-i\epsilon$ and following the cut $(-\infty, \lambda)$ at the (sufficiently small) distance $\epsilon>0$.
Using Corollary \ref{coro:acv}, differentiation under the integral sign is legitimate and we obtain
\begin{equation}\label{pervoe}
s\partial_w{\zeta}(s+1, \Delta^{\tbfm}-\lambda)=\frac{1}{2\pi i} \int_{\Gamma_\lambda}(z-\lambda)^{-s}\sum_{\lambda_n\in \spec(\Delta^{\tbfm})}\frac{-2\partial_w \lambda_n}{(\lambda_n-z)^3}\,dz.
\end{equation}

Using again Corollary \ref{coro:acv}, it is legitimate to make an integration by parts under the integral sign to get
\begin{equation}
s\partial_w{\zeta}(s+1, \Delta^{\tbfm}-\lambda)\,=\,\frac{-s}{2\pi i} \int_{\Gamma_\lambda}(z-\lambda)^{-s-1}\sum_{\lambda_n\in \spec(\Delta^{\tbfm}}\frac{\partial_w \lambda_n}{(\lambda_n-z)^2}\,dz.
\end{equation}

We can now divide by $s$,  use Corollary \ref{coro:acv} once again, and replace $s+1$ by $s$ to finally obtain
\begin{equation}
\partial_w{\zeta}(s, \Delta^{\tbfm}-\lambda)\,=\,\frac{-1}{2\pi i} \int_{\Gamma_\lambda}(z-\lambda)^{-s}\Tr\left(A\partial_zS^{hh}(z)\right)\,dz.
\end{equation}

Using the behaviour  of $S^{hh}$ at infinity we can make an integration by parts again and obtain
\begin{equation}
\partial_w{\zeta}(s, \Delta^{\tbfm}-\lambda)\,=\,\frac{-s}{2\pi i} \int_{\Gamma_\lambda}(z-\lambda)^{-s-1}\Tr\left(A S^{hh}(z)\right)\,dz.
\end{equation}

Differentiating with respect to $s$ and setting $s=0$ gives
\[
\partial_w{\zeta}'(0, \Delta^{\tbfm}-\lambda)\,=\,\frac{-1}{2\pi i} \int_{\Gamma_\lambda}(z-\lambda)^{-1}\Tr\left(A S^{hh}(z)\right)\,dz.
\]
The claim follows by applying Cauchy's theorem.
\end{proof}

Now, using Proposition \ref{relord}, the preceding Proposition and Proposition \ref{clue}  we arrive at the following corollary.

\begin{corollary}\label{varrodin} Let $P_m$ be a zero of the
  meromorphic differential $df$ of multiplicity $\ell_m$ and let
  $z_m=f(P_m)$ be the corresponding critical value of $f$. Let also
  $x_m=(z-z_m)^{\frac{1}{\ell_m+1}}$ be the distinguished local
  parameter in a vicinity of $P_m$. Then
\begin{equation}\label{varodin}
\partial_{z_m}\ln {\rm det}_\zeta (\Delta,\mathring\Delta) %
\,=\,\frac{1}{6(\ell_m+1)(\ell_m-1)!}\left(\frac{d}{dx_m}\right)^{\ell_m-1}S_{Sch}(x_m)\Big|_{x_m=0}\,.
\end{equation}
 \end{corollary}

\section{Integration of the equations for $\ln {\rm Det}$ and explicit expressions for the $\tau$-function}

Let, as before, ${\mathbb B}$ be the matrix of $b$-periods of the
Torelli marked Riemann surface $X$ and let $\{v_\alpha\}_{\alpha=1,
  \dots, g}$ be the basis of the normalized holomorphic differentials on $X$.
Using the Rauch formulas (see, e. g., \cite{KK-IMRN}, \cite{KK-DG}, \cite{KK-MPAG}),
$$\partial_{z_m}{\mathbb B}_{\alpha \beta}=\oint_{P_m}\frac{v_\alpha v_\beta}{df}\,,$$
one immediately gets the
relation
\begin{equation}\label{imB}
\partial_{z_m}\ln {\rm det}\Im {\mathbb B}=\frac{1}{2i}{\rm Tr}\left[(\partial_{z_m}{\mathbb B})(\Im {\mathbb B})^{-1}\right]=
\frac{1}{2i}\oint_{P_m}\frac{\sum_{\alpha \beta}\Im {\mathbb B}^{-1}_{\alpha \beta}v_\alpha v_\beta}{df}\,,
\end{equation}
where the contour integrals are taken over a small contour on $X$ encircling the point $P_m$ (in the positive direction).

Now, using the relation (\ref{connection}), the equations (\ref{imB})
and (\ref{varodin}) together with elementary properties of the
Schwarzian derivative (see, e. g. \cite{Tyurin}),  we arrive at the
following version of Corollary \ref{varrodin} rewritten in the invariant form.
\begin{thm}
Let $P_m$ be a zero of the meromorphic differential $df$ of
multiplicity $\ell_m$ and let $z_m=f(P_m)$ be the corresponding critical
value of $f$. Let also $x_m=(z-z_m)^{\frac{1}{\ell_m+1}}$ be the
distinguished local parameter in a vicinity of $P_m$. Then
\begin{equation}\label{vardwa}
\partial_{z_m}\ln \frac{{\rm det}_\zeta (\Delta,\mathring\Delta)}{{\rm
    det}\Im {\mathbb B}} =-\frac{1}{12\pi i}\oint_{P_m}\frac{S_B-S_f}
{df},
\end{equation}
where $S_B$ is the Bergman projective connection and $S_f=\frac{f'''f'-\frac{3}{2}(f'')^2}{(f')^2}$ is the Schwarzian derivative.
\end{thm}

\begin{remark}Notice that the difference $ S_B-S_f$ is a quadratic differential and, therefore, the integrand in \eqref{vardwa} is a meromorphic one form. 
\end{remark}

It should be noted that the right hand side of (\ref{vardwa}) depends holomorphically on moduli $z_1, \dots, z_M$ and, therefore, one has
$$\partial_{z_m \bar z_n}^2\ln \frac{{\rm det}_\zeta (\Delta,\mathring\Delta)}{{\rm det}\Im {\mathbb B}}=0\,.$$
This implies the relation
\begin{equation}\label{glavotv}
{\rm det}_\zeta(\Delta,\mathring\Delta)=C\,{\rm det}\Im {\mathbb B}\,|\tau|^2\,,\end{equation}
where $\tau$ is a holomorphic function of moduli $z_1, \dots, z_M$
(actually, $\tau$ is a holomorphic section of some holomorphic line
bundle over the Hurwitz space, see \cite{KKZ} for further information;
here we restrict ourselves to local considerations: the reader may
assume for simplicity that all happens in a small vicinity of the
covering $f:X\to {\mathbb C}P^1$ in the Hurwitz space $\Hcal(M, N)$)
subject to the system of PDE
\begin{equation}\label{tau-gov}
\partial_{z_m}\ln \tau=-\frac{1}{12\pi i}\oint_{P_m}\frac{S_B-S_f}
{df}\,
\end{equation}
and $C$ is a moduli independent constant.

System of PDE (\ref{tau-gov}) first appeared in the context of the
theory of isomonodromic deformations and Frobenius manifolds in
\cite{KK-MPAG} and \cite{KK-IMRN}, where, in particular, it was
explicitly integrated. We remind these results in the next subsection.

\subsection{Explicit expressions for $\tau$}
In this section we recall explicit formulas for the holomorphic solution, $\tau$, of the system (\ref{tau-gov}) derived in \cite {KK-IMRN}, \cite{KK-MPAG}
(see also \cite{KS} and \cite{KK-DG} for alternative and more straightforward  proofs).
The result should be formulated separately for low genera $g=0, 1$ and for higher genus $g>1$. We start with the higher genus situation.

Let $g>1$. Take a nonsingular odd theta characteristic $\delta$ and consider the corresponding theta function
$\theta[\delta](t; {\mathbb B})$, where $t=(t_1,\dots,t_g)\in \C^g$. Put
$$\omega_\delta=\sum_{i=1}^g\frac{\del\theta[\delta]}{\del t_i}\left(0; {\mathbb B}\right)\,\omega_i\,.$$
All zeroes of the holomorphic 1-differential $\omega_\delta$ have even multiplicities, and
$\sqrt{\omega_\delta}$ is a well-defined holomorphic spinor on $X$. Following Fay \cite{Fay}, consider
the prime form
\be
E(x,y)=\frac{\theta[\delta]\left(\int_x^y v_1,\dots,\int_x^y v_g; {\mathbb B}\right)}
{\sqrt{\omega_\delta}(x)\sqrt{\omega_\delta}(y)}.
\ee
To make the integrals uniquely defined, we fix $2g$ simple closed loops in the homology classes $a_i,b_i$
that cut $X$ into a connected domain, and pick the integration paths that do not intersect the cuts.
The sign of the square root is chosen so that
$$E(x,y)=\frac{\zeta(y)-\zeta(x)}{\sqrt{d\zeta}(x)\sqrt{d\zeta}(y)}(1+O((\zeta(y)-\zeta(x))^2))$$ as $y\rightarrow x$,
where $\zeta$ is a local parameter such that $d\zeta=\omega_\delta$.

We introduce local coordinates on $X$ that we call {\it distinguished} with respect to $f$.
Consider the divisor $(df)=\sum_k d_k\,p_k,\;p_k\in X,\;d_k\in\Z,d_k\neq 0,$ of the meromorphic differential $df$.
We take $z=f(x)$ as a local coordinate on $X-\bigcup_k p_k$ and
\be
x_k=\begin{cases}\left(f(x)-f(p_k)\right)^{\frac{1}{d_k+1}}&\text{if $d_k>0$,}\\
f(x)^{\frac{1}{d_k+1}}&\text{if $d_k<0$,}\end{cases}\label{np}
\ee
near $p_k\in X$. In terms of these coordinates we have
$E(x,y)=\frac{E(z(x),z(y))}{\sqrt{dz}(x)\sqrt{dz}(y)}$, and we define
\begin{eqnarray*}
E(z,p_k)&=&\lim_{y\rightarrow p_k}E(z(x),z(y))\sqrt{\frac{dz_k}{dz}}(y),\\
E(p_k,p_l)&=&\lim_{\stackrel{\scriptstyle x\rightarrow p_k}{y\rightarrow p_l}}E(z(x),z(y))
\sqrt{\frac{dx_k}{d\zeta}}(x)\sqrt{\frac{dx_l}{d\zeta}}(y)\,.
\end{eqnarray*}
Let ${\mathcal A}^x$ be the Abel map with the basepoint $x$, and let $K^x=(K^x_1,\dots,K^x_g)$ be the vector of Riemann constants
\be
K^x_i=\frac{1}{2}+\frac{1}{2}{\mathbb B}_{ii}-\sum_{j\neq i}\int_{a_i}\left(v_i(y)\int_x^y v_j\right)dy
\label{rc}\ee
(as above, we assume that the integration paths do not intersect the cuts on $X$).
Then we have ${\mathcal A}^x((df))+2K^x=\Omega Z+Z'$ for some $Z,Z'\in\Z^g$ . One has the following expression for the holomorphic solution to
(\ref{tau-gov})(see \cite{KK-IMRN}, here we follow the presentation of this result in \cite{KKZ}):
\be
\tau(X, f)=
\f{\left(\left.\left(\sum_{i=1}^gv_i(\zeta)\frac{\partial}{\partial t_i}\right)^g
\theta(t;{\mathbb B})\right|_{t=K^{\zeta}}\right)^{2/3}}{e^{6^{-1}\pi\sqrt{-1}\langle{\mathbb B} Z+4K^\zeta,Z\rangle}\;W(\zeta)^{2/3}}
\frac{\prod_{k<l}E(p_k,p_l)^{6^{-1}d_k d_l}}
{\prod_k E(\zeta,p_k)^{3^{-1}(g-1)d_k}}\;.
\label{Eg}
\ee
Here $\theta(t;{\mathbb B})=\theta[0](t;{\mathbb B})$ is the Riemann theta function, $t=(t_1,\dots,t_g)\in\C^g$,
and $W$ is the Wronskian of the normalized holomorphic differentials $v_1,\dots,v_g$ on
$X$; the expression in (\ref{Eg}) is independent of $\zeta\in X$.

Let $g=1$. Then the function $\tau(X, f)$ is given (see \cite{KK-MPAG}) by the equation

\begin{equation}\label{E1}
\tau(X, f)=[{\theta_1}'(0\,|\,{\mathbb B})]^{2/3}\frac{\prod_{k=1}^{K} h_j^{(k_j+1)/12}}{\prod_{m=1}^M f_m^{\ell_m/12}}\ \,
,
\end{equation}
where $v(P)$ is the normalized Abelian differential on the elliptic
Torelli marked curve $X$; $v(P)=f_m(x_m)dx_m$ near $P_m$, where
$x_m=(z-z_m)^{1/(\ell_m+1)}$  is the distinguished local parameter
near the zero, $P_m$ of the differential $df$;  $f_m\equiv f_m(0)$;
$v(P)=h_j(\zeta_j)d\zeta_j$ as
$P\to \infty_j$, $\zeta_j=z^{-1/k_j}$, where $k_j$ is the multiplicity
of the pole $\infty_j$ of $f$,   $h_k\equiv h_k(0)$;  $\theta_1$ is
the Jacobi theta-function.

Let $g=0$ and let $U:X\rightarrow \Pl$ be a biholomorphic map such that $U(\infty_1)=\infty$ and $U(P)=(f(P))^{1/k_1}+o(1)$
as $P\to \infty_1$. Then (see \cite{KS})
\begin{equation}\label{E0}
\tau(X, f)=
\frac
{\prod_{j=2}^K(\frac{dU}{d\zeta_j}\big|_{\zeta_j=0})^{(k_j+1)/12}}
{\prod_{m=1}^M(\frac{dU}{dx_m}\big|_{x_m=0})^{l_m/12}}\,.
\end{equation}

Summarizing (\ref{glavotv}) and (\ref{E0},\ref{E1} \ref{Eg}), we get the main result of the present paper.
\begin{thm} Let $(X, f)$ be an element of the Hurwitz space $\Hcal(M, N)$ and let $\tau(X, f)$ be given by expressions (\ref{E0},\ref{E1}, \ref{Eg}).
There is the following explicit expression for the regularized relative determinant of the Laplacian $\Delta$ on the Riemann surface $X$:
\begin{equation}
{\rm det}_\zeta (\Delta,\mathring\Delta)=C\,{\rm det}\Im {\mathbb B}\,|\tau|^2\,,\end{equation}
 where $C$ is a constant dependent only on the connected component of the space $\Hcal(M, N)$ containing the element $(X, f)$.
\end{thm}

\subsection{Examples in genus $0$}

We finish the paper with two simple and especially instructive examples of the calculation of the determinant of the Laplacian $\Delta$ in genus zero.

{\bf Example 1.}
Let $p$ be a polynomial with $N-1$ simple critical points $w_1, \dots, w_{N-1}$ and let the corresponding critical values be $z_1, \dots, z_{N-1}$ (or, what is the same, a ramified covering with $N-1$ simple branch points and one branch point of multiplicity $N$ over the point at infinity of the base.
In other words $p$ is an element of the Hurwitz space  $\Hcal_{0, N}([1]^N)$ of meromorphic functions of degree $N$
on the Riemann sphere $\Pl$ with a single pole of multiplicity $N$.

Let also $w$ be the holomorphic coordinate on the cover $\Pl$  (more precisely, on $\Pl\setminus \{\infty\}$) and $z$ be the holomorphic coordinate on the base $\Pl$. One can assume that the leading coefficient of the polynomial $p(w)$ is equal to one.

Introduce the distinguished local parameter $x_k=\sqrt{z-z_k}$ at $P_k$. Then for $x_k\neq 0$
one has
$$\frac{dw}{dx_k}=\frac{1}{z'(w)}2x_k=\frac{w-w_k}{p'(w)-p'(w_k)}\frac{2x_k}{w-w_k}$$

Passing to the limit $x_k\to 0$, one gets
$$2\left[w'(x_k)\Big|_{x_k=0}\right]^2=\frac{1}{p''(w_k)}$$

Thus,
\begin{equation}\label{chetyre}
\tau=\prod_{k=1}^{N-1}\left[w'(x_k)\Big|_{x_k=0}\right]^{-\frac{1}{12}}=\left\{\prod_{k=1}^{N-1}p''(w_k)\right\}^{\frac{1}{24}}={\cal R}\,(p', p'')^{\frac{1}{24}},\end{equation}
where ${\cal R}\,(f, g)$ is the resultant of polynomials $f$ and $g$ (since the $\tau$-function is defined up to multiplicative constant, the power of $2$ is omitted) and

$${\rm det}_\zeta(\Delta^{|dp|^2}, \mathring\Delta)=C|{\cal R}\,(p', p'')|^{\frac{1}{12}}\,.$$

{\bf Example 2.}

Let $r$ is a rational function with three simple poles (which can be assumed to coincide with $\infty, 0, 1$,
$$r(w)=aw-\frac{b}{w}-\frac{c}{w-1}+d;$$
i. e., $r$ is an element of the Hurwitz space $\Hcal_{0, 3}(1, 1, 1)$ of meromorphic functions on $P^1$ of degree three with three simple poles.

Introducing the local parameter $\zeta=\frac{1}{z}$ in vicinities of the poles $w=0$ and $w=1$ of the cover,
one gets for $\zeta\neq 0$
$$w'(\zeta)=-\frac{1}{r'(w)}r^2(w)$$
and, say, for $w=1$
$$w'(0)=-\lim_{w\to 1}\frac{1}{r'(w)}r^2(w)=-c\,.$$
Analogously, $w'(0)=-b$ at the pole $w=0$.
For the local parameter $\tilde w=w/a$ (for which $\tilde w(P)=z(P)+o(1)$ as $P$ tends to the pole $w=\infty$ of the cover) one has
$\tilde w'(0)=-b/a$ at the pole $w=0$ and $\tilde w'(0)=-c/a$ at the pole $w=1$.
On the other hand writing $r'(w)=\frac{f}{g}$, where $f$ and $g$ are two polynomials and introducing the
local parameter $x_k=\sqrt{z-z_k}$ near the critical point $w_k$ of $r$ ($k=1, 2, 3, 4$),
one gets similarly to (\ref{chetyre})
$$\tilde w'(x_k)=a^{-1}w'(x_k),$$
$$C\prod_{k=1}^4[w'(x_k)|_{x_k=0}]^2=\prod_{k=1}^4\frac{1}{r''(w_k)}=\frac{{\cal R}(f, g)}{{\cal R}(f, f')}$$
Calculating the resultants, one gets
$$\tau^{24}=a^3b^3c^3{\frak M}(a, b, c)\,$$
where
$${\frak M}(a, b, c)=a^3+b^3+c^3+3a^2b+3a^2c+3b^2a+3b^2c+3c^2a+3c^2b-21abc\,$$
and
$${\rm det}_\zeta(\Delta^{|dr|^2}, \mathring\Delta)=C|abc|^{1/4}|{\frak M}(a, b, c)|^{1/12}\,.$$

\end{document}